\documentclass[aap]{imsart}

\usepackage[ruled,linesnumbered]{algorithm2e}
\RequirePackage{amsthm,amsmath,amsfonts,amssymb,todonotes,dsfont,bm}
\RequirePackage[numbers]{natbib}
\RequirePackage[colorlinks,citecolor=blue,urlcolor=blue]{hyperref}
\RequirePackage{graphicx}
\usepackage{tikz}
\usepackage{enumitem}
\usepackage[normalem]{ulem}
\usepackage{comment}
\usetikzlibrary{decorations}
\usetikzlibrary{shapes.geometric}
\usepackage{pgfplots}
\usetikzlibrary{positioning}
\usepackage{subcaption}
\DeclareCaptionLabelFormat{subfig}{\figurename #1~\arabic{figure}.\Alph{subfigure}:}
\newcommand{\per}[0]{\text{per}}
\usepackage{pgfplots}
\pgfplotsset{compat=1.10}
\usepgfplotslibrary{fillbetween}
\usepackage{tikz} 
\usepackage{pgfplots}
\usepackage{forest}
\usetikzlibrary{arrows}
\usetikzlibrary{calc,arrows.meta,positioning}
\tikzset{
    every node/.style={font=\sffamily,minimum size=7mm, inner sep=0pt]},
    main node/.style={circle, draw, minimum size=#1,
              inner sep=3pt, outer sep=1pt},
    matching edge/.style={->,> = latex',red},
    edge/.style={->,> = latex'},
    edges/.style={-}
}

\startlocaldefs

\newtheorem{theorem}{Theorem}[section]
\newtheorem{lemma}[theorem]{Lemma}
\newtheorem{prop}[theorem]{Proposition}
\newtheorem{corollary}[theorem]{Corollary}
\theoremstyle{remark}

\newtheorem{remark}[theorem]{Remark}

\newtheorem*{conjecture}{Conjecture}

\endlocaldefs

\begin{document}

\begin{frontmatter}
\title{sequential importance sampling for estimating expectations over the space of perfect matchings}
\runtitle{SIS over the space of perfect matchings}

\begin{aug}
\author[A]{\fnms{Yeganeh} \snm{Alimohammadi}\ead[label=e1]{yeganeh@stanford.edu}},
\author[B]{\fnms{Persi} \snm{Diaconis}\ead[label=e2,mark]{diaconis@math.stanford.edu}},
\author[A]{\fnms{Mohammad} \snm{Roghani}\ead[label=e4,mark]{roghani@stanford.edu}}
\and
\author[A]{\fnms{Amin} \snm{Saberi}\ead[label=e3,mark]{saberi@stanford.edu}}
\address[A]{Management Science and Engineering,
Stanford University,
\printead{e1,e3,e4}}

\address[B]{Department of Statistics,
Stanford University,
\printead{e2}}
\end{aug}

\begin{abstract}
This paper makes three contributions to estimating the number of perfect matching in bipartite graphs. First, we prove that the popular sequential importance sampling algorithm works in polynomial time for dense bipartite graphs. More carefully, our algorithm gives a $(1\pm\epsilon)$-approximation for the number of perfect matchings of a $\lambda$-dense bipartite graph, using $O(n^{\frac{1-2\lambda}{\lambda}\epsilon^{-2}})$ samples. With size $n$ on each side and for $\frac{1}{2}>\lambda>0$, a $\lambda$-dense bipartite graph has all degrees greater than $(\lambda+\frac{1}{2})n$.  

Second, practical applications of the algorithm requires many calls to matching algorithms. A novel preprocessing step is provided which makes significant improvements. 

Third, three  applications are provided. The first is for counting  Latin squares, the second   is a practical way of computing the greedy algorithm for a card guessing game with feedback, and the third is for stochastic block models. In all three examples, sequential importance sampling allows treating practical problems of reasonably large sizes.

\end{abstract}



\end{frontmatter}

\section{Introduction}\label{sec: intro}

Given a bipartite graph $G(X, Y)$ with $|X|=|Y|=n$, a perfect matching is a subgraph of $G$ with all vertices having degree exactly one. We study the problem of uniform sampling of perfect matchings in bipartite graphs. By classical equivalences between approximate counting and sampling \cite{JERRUM1986169}, uniform sampling is equivalent to approximating the permanent of the adjacency matrix of $G$.

Computing the permanent of a general $0-1$ matrix is $\#$P-complete \cite{VALIANT1979189}. A variety of algorithms have been developed  to approximate the permanent. These are reviewed in Section \ref{sec: background permanents}. A highlight of this development is the Jerrum et al. \cite{Jerrum:2004:PAA:1008731.1008738,jerrum1989approximating} theorem giving a fully polynomial randomized approximation scheme (FPRAS). 
 Alas, despite numerous attempts, they have not been     very effective in practice \cite{newman2020fpras,Sankowski}.

Sequential importance sampling (SIS) constructs a perfect matching sequentially by drawing each edge with weighted probabilities. 
First suggested by \cite{rasmussen}, and actively improved and implemented by \cite{BEICHL1999128, dufosse, Sankowski}. This seems to work well but has eluded theoretical underpinnings. The present paper corrects this.

To state the main result, the following notation is needed.  For an adjacency matrix $A$, define its doubly-stochastic scaling $Q^A$, a matrix with row and column sums equal to $1$, which can be written as
\begin{equation}\label{eq: DS scaling}
    Q^A=D_\alpha A D_\beta,\qquad \text{$D_\alpha$ and $D_\beta$ are diagonal matrices.}
\end{equation} 
For graphs that have at least one perfect matching,  the permanent of $A$ is strictly positive, 
and hence, $Q^A$ exists and it is unique \cite{rothblum1989scalings}.

We guide a sequential importance sampling algorithm using $Q^A$.  The algorithm chooses edges of a perfect matching one at a time, proportionally to their corresponding weights in $Q^A$ (see Algorithm~\ref{alg: SIS-matching} in Section~\ref{sec:dense}).
 Let $(M_1,X_1),\ldots,(M_N,X_N)$ be the outputs of $N$ independent runs of Algorithm \ref{alg: SIS-matching} with $Q^A$ as input, where $M_i$ is a perfect matching and $X_i$ is the probability of sampling $M_i$. To count the number of perfect matchings define the estimator as
\begin{equation}\label{eq: matching estimator}
I_N=\frac{1}{N}\sum_{i=1}^NX_i^{-1}.
\end{equation}
Recall that $G(X,Y)$ is a $\lambda$-dense graph if $|X|=|Y|=n$ and all degrees are greater than $(\lambda+\frac{1}{2})n$. 

\begin{theorem} \label{thm: PTAS dense graph}
Given $\lambda\in(0,\frac{1}{2})$, let $G$ be a $\lambda$-dense  graph of size $2n$. Also, let $M(G)$ be the number of perfect matchings in $G$. Then for any given $\epsilon>0$ there exist constants $C_{\lambda,\epsilon}, C'>0$ such that 
for any  $n\geq C'$,  \[\mathbb{E}\left(\frac{|I_N-M(G)|}{M(G)}\right)\leq \epsilon,\]
where {$N\geq C_{\lambda,\epsilon} n^{\frac{1-2\lambda}{\lambda}\epsilon^{-2}}$, and the running time for obtaining each sample is $O(n^{3})$.}
\end{theorem}

The proof uses the results of Chatterjee and Diaconis \cite{diaconis-chatterjee} to bound the number of samples. 
In general,  sequential importance sampling is used  to approximate a complicated measure $\mu$ with a relatively simple measure $\nu$.
Chatterjee and Diaconis \cite{diaconis-chatterjee} characterized the necessary and sufficient number of samples by the KL-divergence from $\nu$ to $\mu$, when $\log(d\nu/d\mu)$ is concentrated.  In order to prove  Theorem \ref{thm: PTAS dense graph}, we first bound the KL-divergence, and then use their result to find an upper bound on the number of samples.

Our analysis leans on the wonderful paper by Huber and Law \cite{Huber-Law}. They use a completely different approach but `along the way' prove some estimates which are crucial to us. We hope that the approach we take will allow proofs for a wide variety of applications of sequential importance sampling. It would certainly be worthwhile to see also if the approach of Huber and Law can be adapted to some of these problems as well. 
Using our notation, the running time of the algorithm proposed by \cite{Huber-Law} is $O(n^{4}\log(n)+n^{1.5+\frac{1}{4\lambda}})$ and it is  polynomial in $\epsilon^{-1}$.  While our theoretical upper bound for the running time is worse for small values of $\epsilon$, our algorithm may be more widely applicable because it does not directly  rely  on asymptotic bounds on the number of perfect matchings, and it can be applied to generate samples from sparse graphs or calculate any arbitrary statistics of perfect matchings. 

A main contribution of the paper is to show that sequential importance sampling is useful in practice by implementing the algorithm for three fresh applications, as displayed in Section \ref{sec: experiment}. 
The first problem is to estimate the number of Latin squares of order $n$. Asymptotic evaluation of this number is a long open problem. We compare three conjectured approximations using SIS.

The second problem involves a well-studied card guessing game. A deck of $dk$ cards has $d$ distinct values, each  repeated $k$ times. The cards are shuffled and a guessing subject tries to guess the values one at a time. After each guess the subject is told if the guess was right or wrong (not the actual value of the card if wrong). The greedy strategy -- guess the most likely card each time -- involves evaluating permanents. We implement SIS for this problem and give a host of examples and findings.

The third involves counting matchings in bipartite graphs generated by simple stochastic block models. Here, we investigate the benefit of using sequential importance sampling with doubly stochastic scaling (Algorithm \ref{alg: SIS-matching} with $Q^A$ as an input). The simulations show that with $Q^A$ the estimator converges faster to the number of matchings in stochastic block models.

Each of the three problems begins with a separate introduction and literature review. They may be consulted now for future motivations. 
We have prepared a C++ code for scaled sequential importance sampling (Algorithm \ref{alg: SIS-matching}) and its implementation for examples we consider in the paper. You can access it here \cite{ourCode}.

The rest of the paper is organized as follows.  We start by going over some preliminaries  in Section \ref{sec: prelim IS}. Then we proceed by proving Theorem \ref{thm: PTAS dense graph} in Section \ref{sec:dense}.  Section~\ref{sec: experiment} focuses on applications.   Appendix~\ref{sec: fast implmnt SIS} shows the running time of Algorithm \ref{alg: SIS-matching} for generating one sample is  $O(mn)$, where $m$ is the number of edges in $G$. Moreover, it is crucial for us to use $Q^A$ as an input of Algorithm \ref{alg: SIS-matching}.
Appendix \ref{sec:counter example matching} gives an example of a class of graphs with a bounded average degree such that the original method of choosing edges uniformly at random in \cite{diaconis2019randomized} needs an exponential number of samples, while using the doubly stochastic scaling will only need a linear number of samples. Finally, Appendix~\ref{sec: exact permanent card guessing} is on computing the exact permanent of zero-blocked matrices which will be used to find the greedy strategy for card guessing games.

\section{Notation and Background}\label{sec: notation}
In this section, we set up graph notations and give brief background on importance sampling and permanents.

\subsection{Graph Notation}
Given a graph $G$, let $V(G)$ be the set of vertices and $E(G)$ be the set of edges. For a node $v\in V(G)$, let $N(v)$ be the set of neighbors of $v$ in $G$. 

Suppose $G(X,Y)$ is a bipartite graph of size $2n$ with vertex sets $X=\{x_1,\ldots,x_n\}$ and $Y=\{y_1,\ldots, y_n\}$. The adjacency matrix $A$ of a graph $G$ is an $|X|\times |Y|$ binary matrix, where the entry $A_{ij}$ is equal to $1$ if and only if there exists an edge between $x_i$ and $y_j$. Let $S_n$ be the symmetric group of permutations on $n$ elements. We represent a perfect matching by a permutation $\sigma\in S_n$, such that each $x_i$ is matched to $y_{\sigma(i)}$.   Then the permanent of $A$ is defined as
\[\per(A)=\sum_{\sigma\in S_n}\prod_{i=1}^n A_{i\sigma(i)}.\]
For a graph $G$, let $\chi_G$ be the set of all perfect matchings in $G$. 

\subsection{Background on Importance Sampling}\label{sec: prelim IS}
The goal of importance sampling is to estimate the expected value of a function with respect to the probability measure $\nu$ (that is hard to sample) with  samples drawn from a different probability distribution $\mu$ (that is easy to sample). Assume $\mu$ and $\nu$ are two probability measures on the set $\chi$, so that $\nu$ is absolutely continuous with respect to $\mu$. Suppose that we want to evaluate $I(f)=\mathbb{E}_\nu(f)=\int_{\chi}f(x)d\nu(x) $, where $f$ is a measurable function. Let $X_1, X_2,\ldots$ be a sequence sampled from $\mu$ and $\rho=d\nu/d\mu$ be the probability density of $\nu$ with respect to $\mu$. The importance sampling estimator of $I(f)$ is
$$I_N(f)=\frac{1}{N}\sum_{i=1}^N f(X_i)\rho(X_i).$$

In our setting, the  sampling space $\chi$ is the set of perfect matchings,  the target distribution $\nu$ is the uniform distribution over $\chi$, and  $\mu$ is the sampling distribution  by using  Algorithm \ref{alg: SIS-matching}.
Note that here,  $\rho$ is  known only up to a constant. 
Still, we can  use the  estimator $I_N(f)$ to approximate the size of $\chi$. For that purpose, we set $f(x)=|\chi|$ for all $x\in \chi$. This choice of $f$ yields \eqref{eq: matching estimator}. Because for any sample  $X_i$ from the sampling distribution $\mu$, $$\rho(X_i)f(X_i)=\frac{1}{|\chi|\mu(X_i)}f(X_i)=\frac{1}{\mu(X_i)}.$$
Therefore, evaluating $I_N(|\chi|)$ is possible without knowing the exact value of $\rho(X)$ because,
$$I_N(|\chi|)=\frac{1}{N}\sum_{i=1}^N\rho(X_i)f(X_i)=\frac{1}{N}\sum_{i=1}^N\frac{1}{\mu(X_i)}.$$

Given the estimator $I_N$, our goal is to  bound the number of samples $N$. The traditional approach is to bound the variance of the estimator
\begin{equation}\label{variance}
    var(I_N(f))=\frac{1}{N}\Big(\int_{\chi}f(x)^2\rho(x)^2d\mu(x)-I(f)^2 \Big).
\end{equation}
See for example \cite{SiegmundIS,OwenZhu}. 
The number of samples needed for an accurate estimator must be greater than $var(I_N(f))/I(f)^2$ so that the standard deviation of $I_N(f)$ is bounded by $I(f)$.
Here, we follow a different approach which is based on bounding the KL-divergence of the  sampling distribution  ($\mu$) from the target distribution ($\nu$),  defined as $D_{KL}(\nu||\mu)=\mathbb{E}_{\nu}(\log \rho(Y ))$. 
Chatterjee and Diaconis \cite{diaconis-chatterjee} found  a necessary and sufficient number of samples when $\log\rho$ is concentrated with respect to $\nu$. 

\begin{theorem}[Theorem 1.1 in \cite{diaconis-chatterjee}]
\label{thm: num-samples}
Let $\chi, \mu, \nu, \rho, f, I(f)$ and $I_N(f)$ be as above. Let $Y$ be an $\chi$-valued random variable with law $\nu$. Let $L = D_{KL}(\nu||\mu)$ be the Kullback–Leibler divergence of $\mu$ from $\nu$, that is
$$L=D_{KL}(\nu||\mu)=\int_{\chi} \rho(x) \log \rho(x)d\mu(x) =\int_\chi \log \rho(x)d\nu(x) = \mathbb{E}_{\nu}(\log \rho(Y )) .$$
Also, let $||f||_{L^2(\nu)} := (\mathbb{E}(f(Y)^2))^{1/2}$. If
$N = \exp(L + t)$ 
for some $t \geq0$, then
\begin{equation}\label{eq: Chatterjee Diaconis thm num samples}
\mathbb{E}|I_N(f)-I(f)|\leq||f||_{L^2(\nu)}(e^{-t/4}+2\sqrt{\mathbb{P_\nu}(\log \rho(Y)>L+t/2)}).
\end{equation}
Conversely, if $f=1$, $N=\exp(L-t)$ then for any $\delta\in (0,1)$,
$$\mathbb{P}_\nu(I_N(f)>1-\delta)\leq e^{-t/2}+\frac{1}{1-\delta}\mathbb{P}_\nu(\log(Y)<L-t/2).$$
\end{theorem}
The theorem shows that $e^L$ samples are necessary and sufficient for keeping the expected absolute difference of the estimator and its mean small, provided $\log\rho(Y)$ is concentrated around its mean $L$. 

For Theorem \ref{thm: PTAS dense graph}, we need to bound the number of samples required by Algorithm \ref{alg: SIS-matching}. For that purpose, we prove a logarithmic upper bound (in terms of the input size) on $L=D_{KL}(\nu||\mu)$. Then, we use the next lemma to obtain the concentration bounds needed in \eqref{eq: Chatterjee Diaconis thm num samples}. 
\begin{lemma} \label{lm: rho concentration}
With notation as in Theorem \ref{thm: num-samples}, we have
$$\mathbb{P}_\nu\Big(\log\rho(Y)\geq L+t\Big)\leq \frac{L+2}{L+t}.$$
\end{lemma}
\begin{proof}
Markov's inequality implies that
$$\mathbb{P}\Big(\log\rho(Y)>L+t\Big)\leq \frac{\mathbb{E}|\log\rho(Y)|}{L+t}.$$
Note that,
$$\mathbb{E}|\log\rho(Y)|=\mathbb{E}[\log\rho(Y)]+2\mathbb{E}|\log\rho(Y)_-|=L+2\mathbb{E}|\log\rho(Y)_-|,$$
where $\log\rho(Y)_-=\min(0,\log\rho(Y))$. 
It is enough to prove $\mathbb{E}|\log\rho(Y)_-|\leq 1$. One can write
\begin{align*}
    \mathbb{E}|\log\rho(Y)_{-}|&=\int_{0}^{\infty}\mathbb{P}(\log\rho(Y)\leq -x)dx\\
    &=\int_0^{\infty}\mathbb E_\nu[\mathds 1(\rho(Y)\leq e^{-x})] dx & (\mathds 1(.)\text{ is an indicator event})\\
    &=\int_0^{\infty}\mathbb E_\mu[\rho(X)\mathds{1}(\rho(X)\leq e^{-x})] dx. &\text{(change of measure)}
\end{align*}
To finish the proof, note that $\rho(X)\mathds{1}(\rho(X)\leq e^{-x})\leq e^{-x}$. Therefore,
\begin{align*}
    \mathbb{E}|\log\rho(Y)_-|&\leq \int_0^{\infty}e^{-x}dx=1.
    \end{align*}
\end{proof}

\subsection{Background on Permanents}\label{sec: background permanents}
The permanent of a matrix is intimately tied to the world of matchings in bipartite graphs. Fortunately, there is the  bible of Lovasz and Plummer \cite{matching-theory} which thoroughly reviews algorithms and applications to rook polynomials and other areas of combinatorics. For a wide swath of application in statistics and probability problems, see \cite{bapat1990permanents, DiaconisGrahamHolmes}.

Use of what amounts to sequential importance sampling to estimate permanents appears in \cite{rasmussen}. They do not use scaling but do manage to show that the variance of a na\"ive SIS estimator is small for almost all bipartite graphs -- of course, almost all graphs are dense. It is also shown that counting  Hamiltonian circuits is a very similar problem and so should be amenable to present proof techniques. Another application of sequential importance sampling  appears in \cite{kou2009}, which uses non-uniform random permutations generated in cycle format to design an importance sampling estimator.

In \cite{BEICHL1999128, harrissullivan, sullivanBeichl}, Beichl and Sullivan bring in Sinkhorn scaling. Their paper motivates this well and tries it out in a substantial real examples -- dimer coverings of an $n\times n$ square lattice. They forcefully raise the need for theory. We hope that the present paper answers some of their questions. In a later work \cite{jensenBeichl}, they apply scaled importance sampling to counting linear extensions of a partial order.

 There are host of approximation schema that give unbiased estimates of the permanents by computing the determinant of a related random matrix with elements in various fields (or even the Cayley numbers). Barvinok's fine book \cite{barvinok} treats some of these and \cite{Sankowski} gives a survey, tries them out, and employs sequential importance sampling in several numerical examples.
In \cite{bayati2007}, Bayati et al. took a slightly different approach to design a deterministic fully polynomial time approximation scheme for counting matchings in general graphs. They used the correlation decay, which was formerly exploited in several applications such as approximately counting the number of independent sets and colorings in some classes of graphs.

 Another pocket of developments in proving things about sequential importance sampling focuses on specific graphs and gets sharp estimates for the variance and $L$ in Theorem \ref{thm: num-samples} above. This begins with \cite{Diaconis2018SequentialIS}, followed by \cite{ChungPermanent, diaconis2019randomized, tsao2020theoretical}. The last is a thesis with good literature reviews.
 These papers illuminate the pros and cons of  the weights proposed by Sinkhorn scaling. Consider $G$ the `Fibonacci graph' with an edge from $i$ to $i'$ if $|i-i'|\leq 1$, for all $1\leq i,i'\leq n$. This graph has $F_{n+1}$ (Fibonacci number) perfect matchings. Consider building up a perfect matching, adding one edge at a time. Each time there are at most two choices. If the choices are made uniformly, they show that $e^{.002n}$ samples are needed. This is small for reasonable $n$ but still exponential. They prove that if there are two choices, choosing to transpose with probability $p=\frac{3-\sqrt{5}}{2}\sim.382$. This choice gives a sequential importance sampling algorithm which only needs a bounded number of samples, no matter what $n$ is (see \cite{ChungPermanent, diaconis2019randomized}). What weights does Sinkhorn scaling give? We are surprised that it gives $p\sim.333$ for vertices in the middle. From \cite{ChungPermanent}  this choice leads to needing an exponential number of samples. The papers give several further examples with these features. Thus Sinkhorn scaling is good, but it is not perfect.

\section{An SIS Algorithm for Counting the Number of Perfect Matchings}\label{sec:dense}
We start by formalizing the SIS algorithm to sample a perfect matching in general bipartite graphs in Section \ref{sec: alg}. The algorithm starts with an empty set and generates a perfect matching by adding  edges sequentially. At each step, the algorithm keeps a partial matching along with the probability of generating it.

\subsection{Algorithm for General Bipartite Graphs}\label{sec: alg}
Earlier authors \cite{Diaconis2018SequentialIS,diaconis2019randomized,rasmussen} analyzed an importance sampling algorithm that constructs matchings  sequentially by drawing edges uniformly at random. Here, we modify the algorithm by adding the possibility of drawing edges with respect to weights given by an input matrix $Q$.
To formalize the  algorithm, we need the following notation. Given a bipartite graph $G(X,Y)$ with $|X|=|Y|=n$ and a partial matching $M$, call an edge $e$ $M$-\textit{extendable} if there exists a perfect matching in $G$ that contains $M\cup\{ e\}$. 

To sample a perfect matching, take a nonnegative matrix $Q$ as   input, and construct the matching sequentially as follows. First, draw a uniform random permutation $\pi$ over the vertices in $X$. Then start with $M=\emptyset$, an empty matching.
At step $i$, match the $\pi(i)^{th}$ vertex of $X$ according to the following two criteria: 1) The pair must be $M$-extendable. 2) From the extendable edges,  the match is chosen randomly proportional to the weights given by row $\pi(i)$ of $Q$.
After $n$ steps, the algorithm returns a matching $M$, and the probability of generating the matching $p_{\pi,Q}(M)$ with respect to $\pi$ and $Q$. Assume that $Q_e>0$ for all edges $e$, otherwise remove the edge $e$ from graph. Then if the graph has a perfect matching, the algorithm will never fail since by condition 1 we always choose an extendable edge.

\begin{algorithm}[H]
\label{alg: SIS-matching}
\SetAlgoLined
\textbf{Input:} a bipartite graph $G(X,Y)$, and a nonnegative matrix $Q_{n\times n}$.

Draw a uniform random permutation $\pi$ on vertices in $X$.

Set $M=\emptyset$ and $p_{\pi,Q}(M)=1$.

\For{$i$ from 1 to $n$}{
Find the set of extendable neighbors of $x_{\pi(i)}$ (call it $N_{\pi(i)}$).

Find the restriction of row $\pi(i)$ of $Q$ to indices in $N_{\pi(i)}$ (call it $ Q_{\pi(i)}[N_{\pi(i)}]$).

Let $i^*$ be a random index in $N_{\pi(i)}$ drawn with respect to  $Q_{\pi(i)}[N_{\pi(i)}]$.

$M=M\cup{(x_{\pi(i)},i^*)}$.

$p_{\pi,Q}(M)= p_{\pi,Q}(M)\times \frac{ Q_{\pi(i),i^*}}{\sum_{j\in N_{\pi(i)} Q_{\pi(i),j}}}$.
}

\textbf{Input:} $M$, $p_{\pi,Q}(M)$.
\caption{Sequential Importance Sampling of Perfect Matchings}
\end{algorithm}
\begin{remark}
Step 5 of the algorithm might seem challenging to implement. In Appendix \ref{sec: fast implmnt SIS}, we show a fast implementation of it using the Dulmage-Mendelsohn decomposition \cite{matching-theory}.
\end{remark}

To count the number of perfect matchings, let $(M_1,p_{\pi_1,Q}(M_1)),\ldots, (M_N,p_{\pi_N,Q}(M_N))$ be outputs of $N$ runs of Algorithm \ref{alg: SIS-matching} with $Q$ as an input. Then recall the estimator for the number of perfect matchings \eqref{eq: matching estimator},
\[I_N=\frac{1}{N}\sum_{i=1}^N\frac{1}{p_{\pi_i,Q}(M_i)}.\]
The key observation is that $I_N$ is an unbiased estimator on the number of perfect matchings $|\chi_G|$.  This observation is a basic step that has been used in different applications of sequential importance sampling \cite{Diaconis2018SequentialIS,blitzstein2010}.
\begin{prop}
Given a bipartite graph $G$, assume that $Q_e>0$ for all edges $e$.  Define $I_N$ as above. Then
\[\mathbb E [I_N]=|\chi_G|,\]
where the expectation is over the randomness of samples generated by Algorithm \ref{alg: SIS-matching}.
\end{prop}
\begin{proof}
When there is no perfect matching in $G$, the algorithm will find an empty set of extendable edges, first time that it reaches on line 5. Then it always returns $0$, which is an unbiased estimator. For the case that there exists a perfect matching in $G$, the result is immediate by  the linearity of expectation,
\begin{align*}
  \mathbb E [I_N]&=\mathbb{E}[\frac{1}{p_{\pi_1,Q}(M_1)}]
  \\&=\sum_{\pi\in S_n, M\in\chi_G}\frac{1}{n!}p_{\pi,Q}(M)\frac{1}{p_{\pi,Q}(M)}
  \\&=\sum_{\pi\in S_n, M\in\chi_G}\frac{1}{n!}=|\chi_G|.  
\end{align*}
Note that the second equality holds because for any $\pi\in S_n$ and any perfect matching $M$, $p_{\pi,Q}(M)>0$, since by the assumption, $Q_e>0$ for any edge $e$ in $M$.   
\end{proof}


The above proposition proves that the algorithm can be applied to, and gives an unbiased estimator for the number of perfect matchings, for {\em any} bipartite graph $G$.  In what follows, we focus on the case where $G$ is dense and bound the number of samples, proving Theorem \ref{thm: PTAS dense graph}. 
Recall that for some constant $\frac{1}{2}>\lambda>0$, the  bipartite graph $G(X,Y)$ with $|X|=|Y|=n$  is called  $\lambda$-dense if the degree of each vertex is at least $(\frac{1}{2}+\lambda)n$.    For the rest of this section we assume that $G$ is a $\lambda$-dense graph with the adjacency matrix $A$.

Also, note that the variance of the estimator, and the number of samples needed to obtain a certain degree of accuracy, depends on the input to the algorithm $Q$. 
We use the doubly stochastic scaling of the adjacency matrix $Q^A$, 
where
\begin{equation*}
    Q^A=D_\alpha A D_\beta,\qquad \text{$D_\alpha$ and $D_\beta$ are diagonal matrices.}
\end{equation*}
 Finding a $(1-\epsilon)$-approximation of the doubly stochastic matrix is possible in time $\Tilde{O}(m+n^{4/3})$, where $m$ is the number of edges and $\Tilde{O}$ hides a poly-logarithmic factor of $n$ and $1/\epsilon$ \cite{allenzhu2017faster}.

Let $\mu$ be the sampling distribution resulting from Algorithm \ref{alg: SIS-matching} with $Q^A$ as the input.
The following key lemma, which will be used in the proof of Theorem \ref{thm: PTAS dense graph},  gives an upper bound on the KL-divergence from $\mu$ to the uniform distribution.

\begin{lemma} \label{lm: bnd kl-dense}
Given a $\lambda$-dense graph $G$, with $Q^A$ as defined in \eqref{eq: DS scaling}, let $\mu$ be the probability distribution  that a perfect matching is generated by Algorithm \ref{alg: SIS-matching} with $Q^A$ as an input, and $\nu$ be the uniform distribution over the set of perfect matchings in $G$. Then there exists a constant $N_\lambda>0$ such that for $n\geq N_\lambda$,
\[D_{KL}(\nu||\mu)\leq (\frac{1}{8\lambda}-\frac{1}{4})\log n+\frac{1}{\lambda}+\frac{1}{\lambda}\log(\frac{1}{2\lambda}).\]
\end{lemma}

We will give the proof of Lemma \ref{lm: bnd kl-dense}
in Section \ref{sec: kl divergence}. Now we are ready to prove Theorem \ref{thm: PTAS dense graph}.

\begin{proof}[Proof of Theorem \ref{thm: PTAS dense graph}]
Let $L$ be as defined in Theorem \ref{thm: num-samples}. Then by Lemma \ref{lm: bnd kl-dense}
\[L=D_{KL}(\nu||\mu) \leq (\frac{1}{8\lambda}-\frac{1}{4})\log n+\frac{1}{2\lambda}+2\log(\frac{1}{2\lambda}).\]
We apply Theorem \ref{thm: num-samples} for $N=e^{L+t}$, which gives
\[\mathbb{E}(\frac{|I_N-M(G)|}{M(G)})\leq e^{-t/4}+2\sqrt{\mathbb{P}(\log\big(\rho(Y)\big)>L+t/2)}\leq e^{-t/4}+2\sqrt{\frac{L+2}{L+t/2}},\]
where the second inequality is by Lemma \ref{lm: rho concentration}. Now, let $\tilde \epsilon^2=\frac{L+2}{L+t/2}$. {
Then by solving for $t$, we have $t=2\frac{L+2}{\epsilon^{2}}-2L$. This implies, $N=e^{t+L}=C_{\lambda,\tilde\epsilon} n^{\frac{1-2\lambda}{4\lambda}(\tilde\epsilon^{-2}-1)}$} samples are enough for
\[\mathbb{E}(\frac{|I_N-M(G)|}{M(G)})\leq e^{-t/4}+2\sqrt{\frac{L+2}{L+t}}\leq e^{-t/4}+2\epsilon,\]
where 
{$C_{\lambda,\tilde\epsilon}=\exp\Big(\tilde\epsilon^{-2}(\frac{1}{\lambda}+4\log(\frac{1}{2\lambda})+4)-\frac{1}{\lambda}+4\log(\frac{1}{2\lambda})\Big)$ is a constant that is independent of $n$. Note that $e^{-t/4}=o(1)$, and for large enough $n$ it can be bounded by $\tilde\epsilon^3/4$. Now,  by defining $\epsilon =2\tilde \epsilon+\tilde\epsilon^3$, we have $\tilde\epsilon^2-1\leq\epsilon^2/4$,
which proves the upper bound for $N$ in the theorem.}

{To finish the proof, we need to show that the running time for drawing each sample is $O(n^3)$. This is postponed to Proposition~\ref{prop: runtime} in the appendix.}
\end{proof}

As a result of Theorem \ref{thm: PTAS dense graph}, we have a theoretical upper bound on the number of samples needed by SIS to approximate the number of matchings in dense graphs. Algorithm \ref{alg: SIS-matching} can be used beyond estimating the number of perfect matchings. Indeed Lemma \ref{lm: bnd kl-dense} allows us to estimate the expected value of a wide range of statistics under the uniform distribution on perfect matchings.

For a function $f$ over the space of matchings, suppose that we want to estimate $I(f)=\mathbb E_{Y\sim\nu}(f(Y))$, where $\nu$ is the uniform distribution over prefect matchings. As before, let $(M_1,X_1),\ldots, (M_N,X_N)$ be $N$ independent  samples of Algorithm \ref{alg: SIS-matching}. Then define the estimator
\[J_N(f)=\frac{\sum_{i=1}^Nf(M_i)X_i^{-1}}{\sum_{i=1}^N X_i^{-1}}.\]
Recall that $||f||_{L^2(\nu)}:=\big(\mathbb E_{Y\sim \nu}(f(Y)^2)\big)^{1/2}$.
The following gives an upper bound on the sufficient number of samples when  $\frac{||f||_{L^2(\nu)}}{I(f)}$ is bounded by a constant. 
\begin{corollary}\label{cor: statistics bnd}
Given $\lambda>0$, let $G$ be a $\lambda$-dense graph of size $2n$. Also let $\mu$, $\nu$, $f$, $I(f)$, $J_N(f)$, be defined as above. Suppose that there exists a constant $C$ such that $||f||_{L^2(\nu)}\leq C I(f)$. Then given any $\epsilon>0$ there exists a constant $C_\lambda$ such that for $N\geq C_\lambda n^{\frac{16-32\lambda}{\lambda}\epsilon^{-4}}$,
\[\mathbb P \Big( \frac{|J_N(f)-I(f)|}{I(f)}\geq C\epsilon\Big)\leq \epsilon.\]
\end{corollary}
\begin{proof}
By Theorem 1.2 in \cite{diaconis-chatterjee}, if we let  $\epsilon'=(e^{-t/4}+2\sqrt{\mathbb P(\log \rho (Y)>L+t/2)})^{1/2}$, 
then 
\[\mathbb P \Big( |J_N(f)-I(f)|\geq \frac{2||f||_{L^2(\nu)}\epsilon'}{1-\epsilon'}\Big)\leq 2\epsilon'.\]
By Lemma~\ref{lm: rho concentration}, and for $t=128\frac{L+2}{\epsilon^{4}}-2L$
$$\epsilon'\leq \Big(e^{-t/4}+2\sqrt{\frac{L+2}{L+t/2}}\Big)^{1/2}\leq \frac{\epsilon}{2}.$$
Then since $||f||_{L^2(\nu)}\leq CI(f)$,
\[\mathbb P \Big( \frac{|J_N(f)-I(f)|}{I(f)}\geq C\epsilon\Big)\leq \epsilon.\]
Then Lemma \ref{lm: bnd kl-dense} implies the desired upper bound on the number of samples, $N= O(n^{\frac{16-32\lambda}{\lambda}\epsilon^{-4}})$.
\end{proof}

Before proceeding with the rest of the proofs we give two examples of statistics $f$ that satisfy the assumptions in Corollary~\ref{cor: statistics bnd}. 

\textbf{ Example 1 (number of switches):} Consider a random bipartite graph $G$, where each side has $n$ nodes, and edges are drawn independently with a constant probability of $\alpha>1/2$. Note that the average degree of each node is $\alpha n$, and the graph is dense with high probability.
For a perfect matching $M$, define $f(M)$ as the number of pairs of edges $e_1,e_2\in M$ that can be switched to form a new perfect matching. Equivalently, $f(M)$ is the number of $4$-cycles in the graph with exactly two edges in $M$.

We claim that for large enough $n$, $f$ satisfies the assumption of  Corollary~\ref{cor: statistics bnd}. To see this, we need to prove that there exists some $C>0$ such that $\mathbb E[f(M)^2]\leq C\mathbb E[f(M)]^2$.
Define $S_1,S_2\ldots S_{{n\choose 2}^2}$ as the indicator variables for all the possible 4-cycles in a complete bipartite graph.
Then for the perfect matching $M$, $S_i=1$ when all the edges of $S_i$ appear in $G$, with two of those edges appearing in $M$.
By symmetry, $\mathbb{E}[S_i]$ takes the same value for all $i$. Let $q_n$ be this probability,  $q_n=\mathbb{E}[S_1]$. Since with high probability $G$ is dense after removing one edge, by applying the lower bound in Lemma~\ref{lm: P-bound} for both edges of $S_i$ in $M$, one can conclude there exists a constant $c>0$ such that for all $1 \leq i\leq {n\choose 2}^2$,
$$\mathbb{E}[S_i]\geq \frac{c}{n^2}(1+o(1)).$$
Therefore, $\mathbb E[f(M)] = {n\choose 2}^2\mathbb{E}[S_1]$ is at least $C'n^2$ for some constant $C'>0$.

Moreover, 
\[\mathbb E[f(M)^2] = \sum_{i=1}^{{n\choose 2}^2}\mathbb E[S_i]+ \sum_{i=1}^{{n\choose 2}^2}\sum_{j> i}^{{n\choose 2}^2}2\mathbb E[S_i \text{ and } S_j].\]
To compute $\mathbb E[S_i \text{ and } S_j]$, consider two cases. If $S_i$ and $S_j$ have no edges in common, then $\mathbb E[S_i \text{ and } S_j]=\mathbb E[S_i]\mathbb E[S_i| S_j]=q_nq_{n-2}$, for $q_n$ as defined above. The second case is when $S_i$ and $S_j$ has exactly one edge in common. In this scenario, again the remaining graph after removing $S_i$ is dense with high probability and by the upper bound in Lemma~\ref{lm: P-bound}, there exists a constant $c'>0$,
\[\mathbb E[S_i \text{ and } S_j]=\mathbb E[S_i]\mathbb E[S_i| S_j]\leq q_n\frac{c'}{n}(1+o(1))\]
Then by taking into account the number of edges that $S_i$ and $S_j$ have in common,
\[\mathbb E[f(M)^2] \leq {n\choose 2}^2q_n+2{n\choose 4}^2q_nq_{n-2}+6{n\choose 3}^2q_n\frac{c'}{n}(1+o(1)).\]
Thus, $\mathbb E[f(M)^2]$ is upper bounded by $C''n^4$ for some constant $C''>0$ and all large enough $n$.
Now, the claim follows by comparing $\mathbb E[f(M)]^2 $ and $\mathbb E[f(M)^2]$ . 

{\textbf{ Example 2 (colored graph):} Consider $G$ as a dense bipartite graph in which each edge is either red or blue. Define $f(M)$ to be the number of red edges in the perfect matching $M$. We claim that $f$ satisfies the assumption of  Corollary~\ref{cor: statistics bnd} whenever the number of red edges grows at least linearly in $n$. }

To see this, assume the graph has $k$ red edges $e_1,e_2,\ldots,e_k$, and let $X_1,X_2,\ldots X_k$ be the indicator that each red edge appears in  a perfect matching. Then by Lemma~\ref{lm: P-bound},  $\mathbb E[X_i]=\mathbb P(\text{edge $i$ is in a matching})\in(\frac{1}{n},\frac{c}{n})$  for some constant $c>0$ independent from $n$. On the other hand,
\begin{align*}
    \mathbb E[f(M)^2] &= \mathbb E[(\sum_{i=1}^k X_i)^2]\\
    &= \sum_{i=1}^k \mathbb E[X_i^2]+\sum_{i\neq j} \mathbb E[X_iX_j]\\
    &= \sum_{i=1}^k \mathbb P(e_i\in M)+\sum_{i\neq j} \mathbb P(e_i\in M)\mathbb P(e_j\in M|e_i\in M).
\end{align*}
Now, note that after removing $e_i$, the graph is still dense, and again by the use of Lemma \ref{lm: P-bound},  $\mathbb P(e_j\in M|e_i\in M)\in (\frac{1}{n},\frac{c'}{n})$. So, there exists some constant $\alpha>0$ such that $$\mathbb P(e_i\in M)\mathbb P(e_j\in M|e_i\in M)\leq \alpha \mathbb P(e_i\in M)\mathbb P(e_j\in M).$$  Therefore,
\[\mathbb E[f(M)^2]\leq c\frac{k}{n}+c^2\big(\frac{k}{n}\big)^2\alpha.\]
Similarly, we can use the lower bound in  Lemma \ref{lm: P-bound}
\begin{align*}
    \mathbb E[f(M)]^2 &= \big(\sum_{i=1}^k\mathbb E[ X_i]\big)^2\\
    & \geq \big(\frac{k}{n}\big)^2,
\end{align*}
which proves that $f$ satisfies the condition of Corollary~\ref{cor: statistics bnd} when $k$ is at least linear in $n$.

The above example shows that we can estimate the number of red edges in a uniform random perfect matching. This example may be applied in real-world situations, such as matching riders to drivers, where red edges indicate poor matches and blue edges indicate good matches. Then we can evaluate the quality of a given matching by comparing it to the SIS estimator.

\subsection{Bounding the KL-divergence for Dense Graphs}\label{sec: kl divergence}

The purpose of this section is to bound the KL-divergence of $\mu$ from the uniform distribution $\nu$ on the set of perfect matchings. The main idea is to show that if the entries of $Q$ are small, then $D_{KL}(\nu||\mu)$ is a convex function of the entries in $Q$ (see Lemma \ref{lm: KL convex}). Convexity of $D_{KL}(\nu||\mu)$ enables us to use the Bregman's inequality along with the Van der Waerden lower bound on $\per(A)$ to get a logarithmic upper bound on $D_{KL}(\nu||\mu)$.

Let $P$ be the matrix of marginal probabilities, i.e., for an edge $e=(u,v)$, let $P_e= P_{uv}$ be equal to the probability that the edge $e$ appears in a perfect matching chosen uniformly at random. Before proving the convexity of $D_{KL}(\nu||\mu)$,  we need Lemmas \ref{lm: Q-entries} and \ref{lm: P-bound} to show that   $Q_e^A/P_e$ is bounded from above by some constant $c$. 
\begin{lemma}[Lemma 4.4 in \cite{Huber-Law}]\label{lm: Q-entries}
Let $G(X,Y)$ be a $\lambda$-dense bipartite graph with the adjacency matrix $A$, and $Q^A$ its doubly stochastic scaling. Then for all $e\in[n]\times[n]$, $Q^A_e\leq \frac{1}{2\lambda n}$.
\end{lemma}

\begin{lemma}\label{lm: P-bound}
Let $G(X,Y)$ be a $\lambda$-dense bipartite graph and $P$ be the matrix of matching marginals. Then there exists a constant $c_\lambda>0$ independent from $n$ such that for all $e\in E(G)$ we have $\frac{c_\lambda}{n}\leq P_e\leq\frac{1}{\lambda n} $.
\end{lemma}
\begin{proof} To give the proof we first need the following notations.
Given a graph $H$, let $\chi_H$ be the set of all perfect matchings in $H$. For any set $L\subset[n]\times [n]$, define $G_L$ as the graph constructed by removing  vertices appearing in $L$ from $G$.
 Note that $P_e=\frac{|\chi_{G_e}|}{|\chi_G|}$, where $e = (x,y)$. 
 
 We start by proving the lower bound on $P_e$. There are two cases based on whether or not a perfect matching contains $e$. There is an obvious one-to-one correspondence between every matching of $G$ that contains $e$ and matchings in $\chi_{G_e}$. So, we can write, $$|\chi_G| =  |\chi_{G_e}|+|\{M\in \chi_G: e\not\in M\}|.$$ 
 Next, we find a one-to-many mapping $\phi$ between elements of $\{M\in \chi_G: e\not\in M\}$ and $\chi_{G_e}$. Fix a matching $M'\in \chi_G$ that does not contain $e$. Then consider the vertices  matched to the two endpoints of $e=(x,y)$, which we call $M'(x)$ and $M'(y)$, respectively. We claim that there are at least $\lambda n$ edges $(w,z)\in M'$ such that $(w,M'(x)), (M'(y),z)\in E(G)$. To see this, note that  both $M'(x)$ and $M'(y)$ have at least $n/2+\lambda n$ neighbors, and each of these neighbors is matched to a vertex in $M'$. Hence, there are at least $\lambda n$ vertices in common between $N(M'(x))$ and vertices matched to a vertex in $N(M'(y))$. Now, observe that the vertices $z, w, M'(x), x, y,$ and $M'(y)$ form an alternative cycle. As a result, there are $\lambda n$ alternating cycles of length $6$ containing the edge $e=(x,y)$. 

For each such edge $(w,z)$ as above, the mapping $\phi$ maps $M'$ to the perfect matching in $\chi_{G_e}$ formed by removing $e$ and switching the edges  in the alternating cycles containing $(w,z)$.
Formally, $\phi(M')$ defines a set of perfect matchings in $\chi_{G_e}$,
$$\phi(M')=\{\big(M'\setminus\{(x,M'(x)),(y,M'(y)), (w,z)\}\big)\cup \{(w,M'(x)), (M'(y),z)\}\}.$$
By the above arguments, we know that $\phi$ maps each $M'$ to at least  $\lambda n$ elements of $\chi_{G_e}$. Moreover,  each element of $M_e\in\chi_{G_e}$ is the image of at most $n^2$  perfect matchings in $\chi_G$. Because  fixing the vertices $M'(x)$ and $M'(y)$ will determine the reverse image  of each $M_e\in\chi_{G_e}$  uniquely. Since there are at most $n^2$ possible pairs of $M'(x)$ and $M'(y)$,  
$$|\chi_G| =  |\chi_{G_e}|+|\{M\in \chi_G: e\not\in M\}|\leq |\chi_{G_e}|(1+\frac{n^2}{\lambda n}).$$ 
As a result, 
$P_e=\frac{|\chi_{G_e}|}{|\chi_G|}\geq\frac{1}{1+\frac{n}{\lambda}}$.

It remains to prove the upper bound on $P_e$. Recall that $P_e=\frac{|\chi_{G_e}|}{|\chi_G|}$, where $e=(x,y)$.
We prove  $\lambda n\cdot |\chi_{G_e}| \leq |\chi_G|$, by defining a one-to-many mapping that maps every element of $\chi_{G_e}$ to at least $\lambda n$ distinct elements of $\chi_G$.  Fix $M \in \chi_{G_e}$. By an argument similar to the previous bound, there are at least $\lambda n$ edges $(w,z)\in M$ such that $(x,z), (y,w) \in E(G)$.  For each such edge $(w,z)\in M$, note that    $\big(M\setminus\{(w,z)\}\big)\cup \{(y,w), (x,z)\}\in \chi_G$, so map $M$ to  $\big(M\setminus\{(w,z)\}\big)\cup \{(y,w), (x,z)\}$ for every such edge in addition to $M \cup e$. Now, we need to show that at most one element from $\chi_{G_e}$ is mapped to a matching $M' \in \chi_G$. This is because if $M'$ contains the edge $e$, then $M'$ is the image of $M' \setminus e$, and if $e\not\in M'$ then it is the image of $M' \setminus \{(x,M'(x)), (y,M'(y) \}$. As a result, $\lambda n\cdot |\chi_{G_e}| \leq |\chi_G|$. Therefore,
\[P_e=\frac{|\chi_{G_e}|}{|\chi_G|}\leq \frac{1}{\lambda n}.\]
\end{proof}

We will use the next result to relate $p_{\pi,Q}(M)$ to the matching marginals matrix $P$, which will be useful for bounding the KL-divergence. Define the ordering $<_\pi$ on $\{1,2,\ldots,n\}$ as $i<_\pi j$ iff $\pi(i)<\pi(j)$.

\begin{prop}\label{prop: kl reformulation}
Given a bipartite graph $G(X,Y)$ of size $2n$, let $P$ be the matrix of matching marginals and $Q$ be a nonnegative matrix, such that for any $e\in[n]\times[n]$ if $P_e\neq 0$,  then $Q_e\neq 0$. Then for the samples generated by Algorithm \ref{alg: SIS-matching},
\[\mathbb{E}_{M\sim \nu,\pi\sim S_n}\Big(\log(p_{\pi,Q}(M))\Big)\geq\sum_{e\in[n]\times [n]}P_e\log(Q_{e})- \sum_{i=1}^n\sum_{k=1}^nP_{ik}\mathbb{E}_{\pi\sim S_n}\log(\sum_{j\geq_{\pi}k}Q_{i,j}),\]
where the expectation is over $\nu$, the uniform distribution over all perfect matchings in $G$, and the random choice of $\pi$.
\end{prop}
\begin{proof}
The proof is similar to derivation of Equation (4) in \cite{anari2018tight}, except their result is for the case that $Q=P$.
Let $A$ be the adjacency matrix of $G$. Then for a perfect matching $M$ note that 
\[\nu(M)=\frac{1}{\per(A)}.\]
As before, for a matching $M$,  let $M(v)$ be the  vertex matched to $v$ in $M$. The probability that at step $i$ we  match $x_{\pi(i)}$ to $y_{M(\pi(i))}$ is at least $\frac{Q_{\pi(i),M(\pi(i))}}{\sum_{j=i}^nQ_{\pi(i)M(\pi(j))}}$, because the set of extendable neighbors of $x_{\pi(i)}$ is a subset of available vertices at step $i$, i.e., the set $\{y_{M(\pi(i))},y_{M(\pi(i+1))},\ldots, y_{M(\pi(n))}\}$.
Therefore,
\[p_{\pi,Q}(M)\geq \prod_{i=1 }^n\frac{Q_{\pi(i),M(\pi(i))}}{\sum_{j=i}^nQ_{\pi(i)M(\pi(j))}}.\]
As a result,
\begin{align}\label{eq: proposition37eq}
    \mathbb E_{M\sim \nu,\pi\sim S_n}\Big(\log(p_{\pi,Q}(M))\Big)&\geq \mathbb E_{M\sim \nu,\pi\sim S_n}\Big(\log\prod_{i=1 }^n\frac{Q_{\pi(i),M(\pi(i))}}{\sum_{j=i}^nQ_{\pi(i)M(\pi(j))}}\Big)\\
    &=\sum_e P_e\log(Q_e)-\sum_{i=1}^n  \mathbb E_{M\sim \nu,\pi\sim S_n}\Big(\log\big(\sum_{j=i}^n Q_{\pi(i)M(\pi(j))}\big)\Big),
\end{align}
where in the second equality, we used the fact that each edge $e$ appears with probability $P_e$ in a matching drawn from $\nu$.

Next, we use the fact that when $\pi$ is a uniform random permutation, then $M\circ\pi$ is also a uniform random permutation to prove
\begin{equation}\label{eq: new_eq}
    \sum_{i=1}^n\mathbb E_{M\sim \nu,\pi\sim S_n}\Big(\log\big(\sum_{j\geq i} Q_{\pi(i)M(\pi(j))}\big)\Big)=\sum_{i=1}^n \mathbb E_{M\sim \nu,\pi\sim S_n}\Big(\log\big(\sum_{j\geq_\pi M(i)} Q_{ij}\big)\Big).
\end{equation}
 To see this, let $i'=\pi(i)$,
    \[\sum_{j\geq i} Q_{\pi(i)M(\pi(j))}=\sum_{j':\pi^{-1}(j'))\geq i} Q_{i'M(j')}=\sum_{j':\pi^{-1}(j'))\geq \pi^{-1}(i')} Q_{i'M(j')}.\]
    Since  each permutation $\pi$ is reversible, and its reverse appears with the same probability (both with probability $\frac{1}{n!}$),
        \[\sum_{i=1}^n\mathbb E_{M\sim \nu,\pi\sim S_n}\Big(\log\big(\sum_{j\geq i} Q_{\pi(i)M(\pi(j))}\big)\Big)=\sum_{i=1}^n \mathbb E_{M\sim \nu,\pi\sim S_n}\Big(\log\big(\sum_{j\geq_\pi i} Q_{iM(j)}\big)\Big).\]
    Now, since $M$ is a perfect matching, then its reverse $M^{-1}$ is also a perfect matching. Thus,
\begin{align*}
   \sum_{i=1}^n \mathbb E_{M\sim \nu,\pi\sim S_n}\Big(\log\big(\sum_{j\geq_\pi i} Q_{iM(j)}\big)\Big)
    &=\sum_{i=1}^n \mathbb E_{M\sim \nu,\pi\sim S_n}\Big(\log\big(\sum_{M^{-1}(k)\geq_\pi i} Q_{ik}\big)\Big)\\
    &=\sum_{i=1}^n \mathbb E_{M\sim \nu,\pi\sim S_n}\Big(\log\big(\sum_{j\geq_\pi M(i)} Q_{ij}\big)\Big),
\end{align*}
which proves \eqref{eq: new_eq}.

Also, since each edge $e$ is drawn with probability $P_e$ in $M\sim \nu$,
\begin{equation}\label{eq: new_eq2}\sum_{i=1}^n\mathbb E_{M\sim \nu,\pi\sim S_n}\Big(\log\big(\sum_{j\geq_\pi M(i)} Q_{ij}\big)\Big)=\sum_{i=1}^n\sum_{k=1}^n P_{ik}\mathbb E_{\pi\sim S_n}\Big(\log\big(\sum_{j\geq_\pi k} Q_{ij}\big)\Big),
\end{equation}
which gives the statement of the result.
\end{proof}

Define $r_i = n \cdot \max_{1\leq j\leq n}Q^A_{ij}$, for $i \in [n]$. 
Also, define $\mathcal Q$ to be the subset of row-stochastic matrices such that for each $Q\in \mathcal Q$, (i) the entries of row $i$ in $Q$ are at most $\frac{r_i}{n}$, and (ii) each non-zero entry of $Q$ corresponds to a non-zero entry in $Q^A$, i.e. if $Q_{ij}>0$ then $Q_{ij}^A>0$. 
The following result will find a matrix $Q^*\in\mathcal Q$ with  (almost) equal entries in each row, such that roughly speaking, running Algorithm \ref{alg: SIS-matching} on $Q^*$ would result in a larger KL-divergence than running it on $Q^A$. We will see that this is beneficial, because the simple structure of  $Q^*$ will allow us to prove a certain `Bregman-like' upper bound on the KL-divergence corresponding to $Q^*$.

\begin{lemma}\label{lm: KL convex} 
Let $G$ be a $\lambda$-dense graph of size $n$, and let $\mathcal Q$, and $P$ be defined as above. Then  for large enough $n$, there exists
a matrix $Q^*\in\mathcal Q$ such that,
\begin{enumerate}
    \item For each row $i$, all the non-zero entries (except at most one) are equal to $\frac{r_i}{n}$. 
    Equivalently, let $\ell_{1},\ldots \ell_k$ be the indices of non-zero entries in the row $i$ of $Q^*$
    ,  then $Q^*_{i\ell_2}=\cdots =Q^*_{i\ell_k}=\frac{r_i}{n}$.
    \item Further,
\[-\sum_{e\in[n]\times [n]}P_e\log(Q^*_{e})+ \sum_{t,k}P_{tk}\mathbb{E}_{\pi}\log\big(\sum_{j\geq_{\pi}k}Q^*_{tj}\big)\geq \mathbb{E}_{M\sim \nu}\Big(-\log(p_{\pi,Q^A}(M))\Big),\]
where $\nu$ is the uniform distribution on the set of perfect matchings in $G$. 
\end{enumerate}

\end{lemma}
\begin{proof}
 Fix a matrix $Q\in\mathcal Q$ and a row $i$.
Also, fix all the entries of $Q$ except $Q_{ij_1}$ and $Q_{ij_2}$. Let $Q_{ij_1}+Q_{ij_2}=s$. 
By Proposition \ref{prop: kl reformulation}, 
\begin{equation*}
\mathbb{E}_{\sigma\sim \nu,\pi}(-\log(p_{\pi,Q}(\sigma)))\leq -\sum_{e\in[n]\times [n]} P_e\log(Q_{e})+ \sum_{t,k}P_{tk}\mathbb{E}_{\pi}\log(\sum_{j\geq_{\pi}k}Q_{tj}).
\end{equation*}
Define the function $g(x)=-\sum_{e\in[n]\times [n]}P_e\log(Q_{e})+ \sum_{t,k}P_{tk}\mathbb{E}_{\pi}\log(\sum_{j\geq_{\pi}k}Q_{tj})$, where $Q_{ij_1}=x$, $Q_{ij_2}=s-x$ and all other entries of $Q$ are fixed. We will prove that $g(x)$ is a convex function when $ x \in [0,  s]$. But first, let us prove the lemma assuming that is the case. 

Since $g$ is convex, it attains its maximum either at $x=0$ or $x=s$. 
Therefore for any two entries of row $i$ of $Q$ that are not in $\{0,\frac{r_i}{n}\}$, we can increase the value of 

\begin{equation}
\label{eq:upperbound}-\sum_{e\in[n]\times [n]}P_e\log(Q_e)+ \sum_{t,k}P_{tk}\mathbb{E}_{\pi}\log\big(\sum_{j\geq_{\pi}k}Q_{tj}\big),
\end{equation}
by making one of the entries either $0$ or  $\frac{r_i}{n}$ while keeping the other one in the interval $[0,\frac{r_i}{n}]$ and keeping their sum fixed. This operation does not change the row sums and therefore keeps  $Q$ within $\mathcal Q$.  
By repeating this operation we can get a matrix $Q^* \in \mathcal Q$ such that each entry of row $i$ (except at most one) is in $\{0,\frac{r_i}{n}\}$, while only increasing the value of \eqref{eq:upperbound}, hence proving the theorem.

To prove the convexity, we show the second derivative of $g$ is positive. For simplicity, fix $i$ and let $q_j=Q_{ij}$ and $P_j=P_{ij}$. Note that $\sum_{j\geq_\pi k}q_j$ is a constant  when $j_1,j_2\leq_\pi k$ or $j_1, j_2 \geq_\pi k$.
 Therefore,
\begin{align*}
g''(x)= &\frac{P_{j_1}}{x^2}+\frac{P_{j_2}}{(s-x)^2}\\&-\sum_{k\neq j_1,j_2}P_{k}\Big(\sum_{\pi: j_1\geq_\pi k>_\pi j_2}\frac{\mathbb{P}(\pi)}{(\sum_{l\geq_\pi k}q_l )^2}+\sum_{\pi: j_2\geq_\pi k>_\pi j_1}\frac{\mathbb{P}(\pi)}{(\sum_{l\geq_\pi k}q_l )^2}\Big)
\\
&-P_{j_1}\sum_{\pi: j_1>_\pi j_2}\frac{\mathbb{P}(\pi)}{(\sum_{l\geq_\pi j_1}q_l )^2}
-P_{j_2}\sum_{\pi: j_2>_\pi j_1}\frac{\mathbb{P}(\pi)}{(\sum_{l\geq_\pi j_2}q_l )^2}.
\end{align*}
Note that here $\mathbb P(\pi)=\frac{1}{n!}$ is the probability that a uniform permutation is equal to $\pi$.

Since, $x^2<(x+\sum_{\pi:l>_\pi j_1}q_l )^2$ and $\sum_{\pi: j_1>_\pi j_2}\mathbb{P}(\pi)=1/2$, it is enough to prove
\begin{align*}\frac{P_{j_1}}{2x^2}&+\frac{P_{j_2}}{2(s-x)^2}\geq\\& \sum_{k\neq j_1,j_2}P_{k}\Big(\sum_{\pi: j_1\geq_\pi k>_\pi j_2}\frac{\mathbb{P}(\pi)}{(\sum_{l\geq_\pi k}q_l )^2}
+\sum_{\pi: j_2\geq_\pi k>_\pi j_1}\frac{\mathbb{P}(\pi)}{(\sum_{l\geq_\pi k}q_l )^2}\Big).
\end{align*}
Now, since the entries of a permutation over $n-2$ entries of row $i$ are negatively associated
\cite{joag1983negative}, by the Chernoff inequality (e.g., see \cite{dubhashi1996balls}) 
\begin{equation}\label{eq: chernoff}
\mathbb{P}_{\pi}\Big(\sum_{j=1}^tq_{\pi(j)}\leq \frac{\mathbb{E}_\pi[\sum_{j=1}^tq_{\pi(j)}]}{2} \Big)\leq \exp\Big(-\frac{\mathbb{E}_\pi[\sum_{j=1}^tq_{\pi(j)}]^2n^2}{2t}\Big).
\end{equation}
To compute the expectation note that for each $j$, $\mathbb E_{\pi}[q_{\pi(j)}]=\frac{1}{n}$ because $Q$ is row-stochastic. Hence, by the linearity of expectation 
 $ \mathbb{E}_\pi(\sum_{j=1}^tq_{\pi(j)})=\frac{t}{n}$. Further, observe that $\mathbb{P}(\pi(t)=k, j_1<_\pi k\leq_\pi j_2 )\leq \frac{t}{n(n-1)}$. Then
\begin{align*}
    &\sum_{\pi: j_1 \geq_\pi k>_\pi j_2}\frac{\mathbb{P}(\pi)}{(\sum_{l\geq_\pi k}q_l )^2}
    \\
    &\qquad= \sum_{t=1}^n\mathbb{P}(\pi(t)=k, j_1<_\pi k\leq_\pi j_2 )\mathbb E_\pi\Big[\frac{1}{(\sum_{l\geq_\pi k}q_l )^2}\mid \pi(t)=k, j_1<_\pi k\leq_\pi j_2\Big]\\
    &\qquad\leq \sum_{t=1}^{\lceil\log n\rceil} \frac{t}{n(n-1)x^2}\\
    &\qquad\qquad+\sum_{t=\lceil\log n\rceil+1}^{n} \frac{t}{n(n-1)}\Big(\frac{1}{(\mathbb{E}_\pi[\sum_{j=1}^tq_{\pi(j)}]/2)^2}+\frac{\mathbb{P}_{\pi}\Big(\sum_{j=1}^tq_{\pi(j)}\leq \frac{\mathbb{E}_\pi[\sum_{j=1}^tq_{\pi(j)}]}{2} \Big)}{x^2}\Big)\\
    &\qquad\leq \sum_{t=1}^{\lceil\log n\rceil} \frac{t}{n(n-1)x^2}+\sum_{t=\lceil\log n\rceil+1}^{n} \frac{t}{n(n-1)}\Big(\frac{1}{(x+\frac{t-1}{2n})^2}+\frac{e^{-t/2}}{x^2}\Big)\\
    &\qquad\leq \frac{\log^2 n}{n(n-1)x^2}+ 4\log(4xn+n)+\frac{2\log n}{n(n-1)x^2},
\end{align*}
where in the first inequality we used \eqref{eq: chernoff} and the fact that $\sum_{l\geq_\pi k}q_l\geq x$ when $j_1\geq_\pi k$.
 A similar inequality holds  for $s-x$. Note that for large enough $n$, we have $2\log n\leq \log^2 n$. Therefore, 
 \begin{align*}
   & g''(x)\geq\\
   &\frac{P_{j_1}}{2x^2}+\frac{P_{j_2}}{2(s-x)^2}-\frac{2\log^2 n}{(nx)^2}-4\log{(4xn+n)}-\frac{2\log n}{(n(s-x))^2}-4\log{(4(s-x)n+n)}.
 \end{align*} 
Note that for $n\geq (4r_i+1)^4$ we have
\[\frac{\log^2 n}{(nx)^2}\geq\frac{\log^2{n}}{(r_i)^2}\geq 4\log(4r_i+n)\geq4\log(4xn+n).\] 
As a result,
 \begin{align*}
    g''(x)\geq
    &\frac{P_{j_1}}{2x^2}+\frac{P_{j_2}}{2(s-x)^2}-\frac{3\log^2 n}{(nx)^2}-\frac{3\log^2 n}{(n(s-x))^2}.
 \end{align*} 
By Lemma \ref{lm: P-bound} and the fact that $x,s-x\in[0,\frac{r_i}{n}]$ there exists $C>0$ such that both $P_{j_1}$ and $P_{j_2}$ are  at least $C/n$.
Therefore, $g''(x)\geq 0$ for $n\geq\max\Big(3C^{-1},(4r_i+1)^4,e^2\Big)$, which proves the convexity of $g$. 
\end{proof}
Now, we are ready to bound the KL-divergence.

\begin{proof}[Proof of Lemma \ref{lm: bnd kl-dense}]
Let $\pi$ be the permutation that is chosen by Algorithm \ref{alg: SIS-matching}. Since $\nu$ assigns probability $\frac{1}{\per(A)}$ to each perfect matching, 
\[D_{KL}(\nu||\mu)=\mathbb{E}_{M\sim \mu}\Big(-\log(p_{\pi,Q^A}(M))\Big)-\log(\per(A)).\]
Let $Q^A_i$ denote the $i^{th}$ row of $Q^A$ and let $\frac{r_i}{n}$ be the maximum entry of $Q^A_i$.  
First, we use Lemma~\ref{lm: KL convex} to upper bound the first term: 
\[\mathbb{E}_{M\sim \mu}\Big(-\log(p_{\pi,Q^A}(M))\Big) \leq -\sum_{e\in[n]\times [n]}P_e\log(Q^*_{e})+ \sum_{i,k}P_{ik}\mathbb{E}_{\pi}\log(\sum_{j\geq_{\pi}k}Q^*_{ij}),\]
where in $Q^*$ row sums are equal to $1$  and  each entry of $Q^*_i$ (except at most one) is in $\{0,\frac{r_i}{n}\}$.
Then, we prove

\begin{equation}\label{eq: bregman}
-\sum_{e\in[n]\times [n]}P_{e}\log(Q^*_{e})+ \sum_{i,k}P_{ik}\mathbb{E}_{\pi}\log(\sum_{j\geq_{\pi}k}Q^*_{ij})\leq\log\prod_{i=1}^n\Big( \big(\lfloor\frac{n}{r_i}\rfloor+1\big)!^{1/(\lfloor\frac{n}{r_i}\rfloor+1)}r_i^{\frac{1}{\lambda n}} \Big)
\end{equation}
and
\begin{equation}\label{eq: perm}
\log(\per(A)) \geq \log\left(e^{-n}\sqrt{2\pi n}\prod_i \frac{n}{r_i}\right).
\end{equation}

Assuming the above inequalities are true, we will finish the proof of the lemma. By combining inequalities  (\ref{eq: bregman}) and (\ref{eq: perm}), and using Stirling's approximation in \eqref{eq: bregman},
\begin{align*}
    D_{KL}(\nu||\mu) &\leq \log\Big(\frac{\prod_{i=1}^n\big(\sqrt{2\pi(\lfloor\frac{n}{r_i}\rfloor+1)}^{\frac{r_i}{n}}\big(\lfloor\frac{n}{r_i}\rfloor+1\big)e^{r_i/(12n^2)-1} r_i^{\frac{1}{\lambda n}}\big)}{(\prod_{i=1}^n \frac{n}{r_i} )e^{-n}\sqrt{2\pi n}}\Big)\\
    &\leq \sum_{i=1}^n\Big(\frac{r_i}{2n}\log( \frac{1}{r_i}+\frac{1}{n})+\log(1+\frac{r_i}{n})+\frac{1}{\lambda n}\log(r_i)\Big)\\
    &\qquad\qquad+(\sum_{i=1}^n \frac{r_i}{2n})\big(\log(2\pi n)+\frac{1}{6n}\big)-\frac{1}{2}\log(2\pi n)\\
    &\leq  (\frac{1}{2n}\sum_{i=1}^n r_i-\frac{1}{2})\log(2\pi n)+\frac{1}{12n^2}\sum_{i=1}^nr_i+\frac{1}{2\lambda}+\frac{1}{\lambda}\log(\frac{1}{2\lambda}),
\end{align*}
where in the last inequality we used the facts that $r_i\leq(2\lambda)^{-1}$ by Lemma \ref{lm: Q-entries}, and that for large enough $n$, we have $\log(\frac{1}{r_i}+\frac{1}{n})\leq \log(2\lambda+\frac{1}{n})\leq 0$, and $\log(1+\frac{r_i}{n})\leq \frac{r_i}{n}\leq (2\lambda n)^{-1}$.

In order to bound $\sum_i r_i$, assume  without loss of generality  that $\beta_1=\max_j(\beta_j)$. If $A_{i1}$ is non-zero in row $i$, then $Q^A_{i1}$ is equal to $\beta_1 \alpha_i$. Furthermore, it is the highest element in row $i$ and therefore   $\frac{r_i}{n} = \beta_1 \alpha_i $. Also, since $Q^A$ is doubly stochastic,  $\beta_1 \sum_{i\sim 1}\alpha_i \leq 1$ where $i\sim 1$ denotes that $A_{i1}\neq 0$. Therefore, by writing the first column sum, 
$$\frac{1}{n}\sum_{i\sim 1} r_i=\beta_1(\sum_{i\sim 1}\alpha_i)\leq 1.$$
 For the  rows $i\not\sim1$,   use Lemma \ref{lm: Q-entries} which implies $r_i\leq \frac{1}{2\lambda}$. Therefore,
$$\frac{1}{n}\sum_{i=1}^nr_i=\frac{1}{n}\sum_{i\sim1}r_i+\frac{1}{n}\sum_{i\not\sim1}r_i\leq 1+n(1/2-\lambda)\frac{1}{2\lambda n}=\frac{1}{2}+\frac{1}{4\lambda},$$
which finishes the proof, since
\begin{align*}
 D_{KL}(\nu||\mu)&\leq (\frac{1}{2n}\sum_{i=1}^n r_i-\frac{1}{2})\log(2\pi n)+\frac{1}{12n}+\frac{1}{2\lambda}+\frac{1}{\lambda}\log(\frac{1}{2\lambda})\\&\leq (\frac{1}{8\lambda}-\frac{1}{4})\log n+\frac{1}{\lambda}+\frac{1}{\lambda}\log(\frac{1}{2\lambda}),
\end{align*}
for $n\geq \lambda/2$.

It remains to prove (\ref{eq: perm}) and (\ref{eq: bregman}). First, we prove (\ref{eq: perm}). Recall that $Q^A=D_\alpha A D_\beta$, where $D_\alpha$ and $D_\beta$ are diagonal matrices with entries
$\alpha_1,\ldots,\alpha_n$, and $\beta_1,\ldots,\beta_n$, respectively.
Without loss of generality, assume there exists a matching between $i$ of $X$ to the vertex $i$ of $Y$ (otherwise we can reorder vertices). Then $\alpha_i\beta_i A_{ii}=Q^A_{ii}\leq \frac{r_i}{n}$ by the definition of $r_i$. By this observation and the Van der Waerden inequality for $Q^A$ (see e.g., \cite{knuth_permanent}),
$$\per(A)=\per(D_\alpha^{-1})\cdot \per(Q^A)\cdot \per(D_\beta^{-1})\geq (\prod_i \frac{1}{\alpha_i\beta_i}) \frac{n!}{n^n}\geq e^{-n}\sqrt{2\pi n}\prod_i \frac{n}{r_i}.$$

Next, we prove (\ref{eq: bregman}). For an $n\times 1$ vector $\bm{c}$, let $\bm{c}\cdot Q^*$ be a matrix with each row of $Q^*$ multiplied by a constant factor. 
Note that the sampling distribution of Algorithm \ref{alg: SIS-matching} does not change, i.e., for any matching $M$, $p_{Q^*}(M)=p_{\bm{c}\cdot Q^*}(M)$. So, by choosing the right constants we can assume each row, $Q^*_i$, has $\lfloor\frac{n}{r_i}\rfloor$ entries equal to $1$ and all other entries (except possibly one entry) is equal to $0$. Note that this normalization does not change the  left hand side of the inequality  (\ref{eq: bregman}). Indeed, following the proof of (\ref{eq: new_eq}-~\ref{eq: new_eq2})
in Proposition \ref{prop: kl reformulation},
\begin{align*}
    \sum_{(i,j)\in[n]\times [n]}P_{(i,j)}\log(\bm{c}_i\cdot Q^*_{(i,j)})&- \sum_{i,k}P_{ik}\mathbb{E}_{\pi}\log(\sum_{j\geq_{\pi}k} \bm{c}_i\cdot Q^*_{ij}) \\&=
    \mathbb E_{M\sim \nu,\pi\sim S_n}\Big(\log\prod_{i=1 }^n\frac{\bm{c}_{\pi(i)}Q_{\pi(i),M(\pi(i))}}{\sum_{j=i}^n\bm{c}_{\pi(i)}Q_{\pi(i)M(\pi(j))}}\Big)\\&= \mathbb E_{M\sim \nu,\pi\sim S_n}\Big(\log\prod_{i=1 }^n\frac{Q_{\pi(i),M(\pi(i))}}{\sum_{j=i}^nQ_{\pi(i)M(\pi(j))}}\Big) \\&=
     \sum_{e\in[n]\times [n]}P_e\log(Q^*_{e})- \sum_{i,k}P_{ik}\mathbb{E}_{\pi}\log(\sum_{j\geq_{\pi}k}Q^*_{ij}).
\end{align*}

So it is sufficient to show that
\begin{equation}\label{eq: bregman1}
\begin{aligned}
-\sum_{(i,j)\in[n]\times [n]}P_{(i,j)}\log(\bm{c}_i\cdot Q^*_{(i,j)})+ \sum_{i,k}P_{ik}\mathbb{E}_{\pi}\log(\sum_{j\geq_{\pi}k}\bm{c}_i\cdot Q^*_{ij})\\
\quad\leq\log\prod_{i=1}^n\Big( \big(\lfloor\frac{n}{r_i}\rfloor+1\big)!^{1/(\lfloor\frac{n}{r_i}\rfloor+1)}r_i^{\frac{1}{\lambda n}} \Big).
\end{aligned}
\end{equation}

We now prove the final inequality.
By the construction of $Q^*$ in Lemma \ref{lm: KL convex}, all entries (except at most one) in each row of $\bm{c}\cdot Q^*$ are equal to $1$. The smallest non-zero entry smaller than one in  row $i$  of $\bm{c}\cdot Q^*$ is equal to $\frac{n}{r_i}-\lfloor\frac{n}{r_i}\rfloor$. So, if such an entry outside of $\{0,1\}$ exists, i.e., if $n$ is not divisible by $r_i$, then the smallest non-zero entry is at least $\frac{1}{r_i}$. Further, the matching marginal of each edge cannot be more than $\frac{1}{\lambda n}$, by Lemma \ref{lm: P-bound}.  
Therefore,
\[-\sum_{(i,j)\in[n]\times [n]}P_{(i,j)}\log(\bm{c}_i\cdot Q^*_{(i,j)})\leq \sum_{i=1}^n\frac{1}{\lambda n}\log(r_i). \]

Now, we need to bound the second term in the left hand side of \eqref{eq: bregman1}. Let $\bar Q^*$ be a binary ($0-1$) matrix that each of its entries are equal to $1$ if and only if the same entry is non-zero in $Q^*$. Note that $\bar Q^*$ is equal to $\bm{c}\cdot Q^*$ except at most one entry per row. Then \[\sum_{i,k}P_{ik}\mathbb{E}_{\pi}\log(\sum_{j\geq_{\pi}k} \bm{c}_i\cdot Q^*_{ij})\leq \sum_{i,k}P_{ik}\mathbb{E}_{\pi}\log(\sum_{j\geq_{\pi}k}\bar Q^*_{ij}).\]

The rest is an upper bound similar to Bregman-Minc's inequality \cite{MincPermanent,bregman1973some}.
Note that 
$\sum_{j\geq_{\pi}k}\bar Q^*_{ij}$ is equal to the number of non-zero entries of row $i$ that appear after $k$.  There are $\lfloor\frac{n}{r_i}\rfloor+1$   non-zero entries and the probability that $k$ appears before $\ell$ number of them  in $\pi$ is equal to $\frac{1}{\lfloor\frac{n}{r_i}\rfloor+1}$. Therefore,
\begin{align*}
    \sum_{i,k}P_{ik}\mathbb{E}_{\pi}\log(\sum_{j\geq_{\pi}k}\bar Q^*_{ij})&=\sum_{i,k}P_{ik}\frac{1}{\lfloor\frac{n}{r_i}\rfloor+1}\sum_{\ell=1}^{\lfloor\frac{n}{r_i}\rfloor+1}\log(\ell)\\
    &=\sum_{i,k}P_{ik}\frac{1}{\lfloor\frac{n}{r_i}\rfloor+1}\log((\lfloor\frac{n}{r_i}\rfloor+1)!)\\
     &=\sum_{i}\frac{1}{\lfloor\frac{n}{r_i}\rfloor+1}\log((\lfloor\frac{n}{r_i}\rfloor+1)!),
\end{align*}
where the last equality is because the sum of matching marginals over each row is equal to $1$. As a result, we have proved \eqref{eq: bregman1}.
\end{proof}

In Section \ref{sec: experiment}, three different applications are given for SIS for sampling perfect matchings. Although some of our applications include situations where the underlying graph is not dense, we will see that SIS with doubly stochastic scaling still converges rapidly in our simulations. It remains open to give theoretical bounds, supporting the simulation results, on the convergence of SIS in graphs that are not dense.

\section{Applications}\label{sec: experiment}

Estimating the number of perfect matchings has applications in various settings. In this section, we present simulation results of three applications of SIS with doubly stochastic scaling. 

We start by counting the number of Latin squares and rectangles. 
As  will be discussed in Section \ref{sec: latin square}, a $k\times n$ Latin rectangle corresponds to $k$ disjoint perfect matchings in a complete bipartite graph of size $n$. Using SIS, we sample and estimate the number of Latin squares. Then, we use an SIS estimator as a benchmark to test three conjectures on the asymptotic number of Latin squares.

We continue by applying SIS to card guessing experiments in Section \ref{sec: card guessing SIS}. A deck of $n$ cards is shuffled and one has to guess the cards one by one. We will see that guessing the most likely card at each step (the greedy strategy) reduces to evaluating permanents. Therefore, we can apply SIS to find the number of correct card guesses using greedy for large decks of cards. Note that in the literature, the exact expectation of correct guesses using greedy was only known for small deck sizes (see e.g., \cite{diaconis2020card,diaconis2020guessing}).

Finally, Section \ref{sec: SIS for SBM} demonstrates the importance of using the doubly stochastic scaling of the adjacency matrix as the input of SIS. We compare SIS with and without doubly stochastic scaling to count the number of perfect matchings in bipartite graphs generated by stochastic block models. As we will see, using the doubly stochastic scaling of the adjacency matrix can make the standard deviation $4$ times lower in some cases.
\subsection{Counting Latin Rectangles}\label{sec: latin square}
The first application of sequential importance sampling is for counting the number of Latin rectangles.
An $n\times n$ Latin square is an $n\times n$ matrix with entries in $\{1,\ldots,n\}$, such that each row and each column contains distinct integers. Let $L_{n}$ be the number of $n\times n$ Latin squares. A $k\times n$ Latin rectangle is a $k$ by $n$ array with all rows containing $\{1,2,...,n\}$ and all columns distinct

The exact values of $L_{k,n}$ is only known for small  $k$ and $n$. Indeed, $L_{1,n}=n!$. Also, the number of ways to fill out the second row of a $2\times n$ Latin square is equal to the number of derangements of $\{1,\ldots,n\}$, leading to $L_{2,n}\sim \frac{(n!)^2}{e}$.  A series of classical works give the asymptotics of $L_{k,n}$ when $k=o(n^{1/2})$ \cite{ErdosKaplansky,yamamoto1952asymptotic,Yamamoto,STEIN197838}.
Godsil and McKay \cite{GODSIL199019} generalized these results to the case $k=o(n^{6/7})$, with  the following asymptotics $$L_{k,n}\sim\frac{(n!)^k([n]_k)^n}{e^{k/2}(1-\frac{k}{n})^{n/2}n^{kn}}.$$

There are Markov chains on the space of Latin squares with uniform stationary distributions \cite{jacobson1996generating, PITTENGER1997251}. Alas, at this writing, there are no known bounds on the mixing time. Another method is to use divide and conquer to generate an exact uniform sample \cite{desalvo2017exact}. 
Here, we use sequential importance sampling to generate a Latin rectangle row by row. First, we describe the algorithm to sample a Latin square. Then we compare our estimator with known exact values and an earlier versions of importance sampling given by Kuznetsov \cite{kuznetsov2009estimating}. Then we test three conjectured asymptotics on the number of Latin rectangles and squares. Finally, we test a conjecture by Cameron \cite{cameronschool} on the number of odd permutation rows in a typical Latin square.

To describe the SIS algorithm, let $G(X,Y)$ be a bipartite graph, where  $X$ represents entries in  a row and $Y=\{1,\ldots,n\}$ represents the possible values for each entry. Then we sample a Latin rectangle row by row. Start with $G=K_{n,n}$ and repeat the following for $k$ steps: Sample a perfect matching with sequential importance sampling and then remove its edges from $G$.
This procedure is repeated until a Latin rectangle is obtained. In the experiment below, this is repeated $N=10^7$ times resulting in Latin rectangles, each with a weight equal to the product of importance sampling weights for all rows.
Define the estimator $L_{k,n}^{SIS}$ as the average of the $N$ weights of $k \times n$ Latin rectangles. In the same way, we define the estimator $L_n^{SIS}$ for the number of Latin squares.

First, we compare $L_n^{SIS}$ with some exact values of $L_n$ in Table \ref{tbl: latin square}. Note that the exact values of $L_n$ are only known up to $n=11$. See \cite{Stones} for a comprehensive recent review on computing the value of  $L_{k,n}$. To simplify the reported numbers, we divided $L_n$ and $L_n^{SIS}$ by $n!(n-1)!$, since there are $n!(n-1)!$ ways to generate the first row and the first column of a Latin square. As shown in the table, the relative error of SIS estimator is less than $0.1\%$ for all $5\leq n\leq 11$. Also, an early version of unscaled importance sampling with rejections  has been given by Kuznetsov \cite{kuznetsov2009estimating}. In Table \ref{tbl: latin square kuznetsov}, we compare our estimator with Kuznetsov's estimator.
\begin{table}[!htbp]
 \begin{tabular}{||c|| c| c| c |c|c|c||} 
 \hline
 \rule{0pt}{15pt}
 $n$ & Runs & $\frac{L_n^{SIS}}{n!(n-1)!}$ & Confidence Interval & $\frac{L_n}{n!(n-1)!}$ & \% error & References  \\ [0.5ex] 
 \hline\hline
$5$ & $10^7$ & $56.021$ & $(56.000, 56.041)$ & $56$ & $0.0375$ &\\ 
 \hline
 6 & $10^7$  & $9406.3$ & $(9400.9, 9411.7)$ & $9408$ & $0.0181$ &\\
 \hline
 7 & $10^7$ & $1.6945\times 10^7 $& $(1.6933\times 10^7 , 1.6958\times 10^7)$ & $1.6942\times 10^7$ & $0.0177$ &\\
 \hline
 8 & $10^7$ & $5.3529 \times 10^{11}$ & $(5.3475\times 10^{11}, 5.3583\times 10^{11})$ & $5.3528\times 10^{11}$ & $0.0019$ & \cite{mullen1993some,wells1967number} \\
 \hline
 9 & $10^7$ & $3.7781 \times 10^{17}$ & $(3.7729\times 10^{17}, 3.7834\times 10^{17})$ & $3.7759\times 10^{17}$ & $0.0583$ & \cite{mullen1993some} \\ 
  \hline
 10& $10^7$  & $7.5876 \times 10^{24}$ & $(7.5730\times 10^{24}, 7.6024\times 10^{24})$ & $7.5807\times 10^{24}$ & $0.0910$ & \cite{mckay1995latin} \\ 
  \hline
 11& $10^7$  & $5.3687 \times 10^{33}$ & $(5.3539\times 10^{33}, 5.3836\times 10^{33})$ & $5.3639\times 10^{33}$ & $0.0895$ & \cite{mckay2005number} \\ [1ex] 
 \hline
\end{tabular}
\caption{Comparison of our SIS estimator ($L_n^{SIS}$) with the exact values of $L_n$ for $5\leq n\leq 11$.} 
\label{tbl: latin square}
\end{table}

\begin{table}[!htbp]
 \begin{tabular}{||c|| c| c|c |c|c||} 
 \hline
 \rule{0pt}{15pt}
 $n$ & Runs & $L_n^{SIS}$  & Kuznetsov's estimator& The exact value of $L_n$\\ [0.5ex] 
 \hline\hline
$5$ & $10^7$ & \begin{tabular}{@{}c@{}}$1.613 \times 10^5$\\$(1.613\times10^5, 1.614 \times 10^5)$ \end{tabular} &  \begin{tabular}{@{}c@{}} $1.609 \times 10^5$\\$(1.593\times10^5, 1.625 \times 10^5)$\end{tabular}& 161280\\ 
 \hline
 6 & $10^7$  & \begin{tabular}{@{}c@{}}$8.127 \times 10^8$\\$(8.122\times10^8, 8.132 \times 10^8)$\end{tabular} & \begin{tabular}{@{}c@{}}$8.135 \times 10^8$\\$(8.054\times10^8, 8.217 \times 10^8)$\end{tabular} & $8.129 \times 10^8$\\
 \hline
 7 & $10^7$ & \begin{tabular}{@{}c@{}}$6.149 \times 10^{13}$\\$(6.145\times10^{13}, 6.153 \times 10^{13})$\end{tabular} & \begin{tabular}{@{}c@{}}$6.149 \times 10^{13}$\\$(6.087\times10^{13}, 6.210 \times 10^{13})$\end{tabular} & $6.148 \times 10^{13}$\\
 \hline
 8 & $10^7$ & \begin{tabular}{@{}c@{}}$1.088 \times 10^{20}$\\$(1.087\times10^{20}, 1.089 \times 10^{20})$\end{tabular} &  \begin{tabular}{@{}c@{}}$1.095 \times 10^{20}$\\$(1.084\times10^{20}, 1.106 \times 10^{20})$\end{tabular}& $1.088 \times 10^{20}$\\
 \hline
 9 & $10^7$ & \begin{tabular}{@{}c@{}}$5.528 \times 10^{27}$\\$(5.520\times10^{27}, 5.536 \times 10^{27})$\end{tabular} & \begin{tabular}{@{}c@{}}$5.531 \times 10^{27}$\\$(5.475\times10^{27}, 5.586 \times 10^{27})$\end{tabular} & $5.525 \times 10^{27}$\\ 
  \hline
 10& $10^7$  & \begin{tabular}{@{}c@{}}$9.992 \times 10^{36}$\\$(9.972 \times10^{36},10.011 \times 10^{36})$\end{tabular} &  \begin{tabular}{@{}c@{}}$9.991 \times 10^{36}$\\$(9.891 \times10^{36},10.091 \times 10^{36})$\end{tabular}& $9.982 \times 10^{36}$ \\ 
  \hline
 11& $10^7$  & \begin{tabular}{@{}c@{}}$7.777 \times 10^{47}$\\$(7.755\times10^{47}, 7.798\times 10^{47})$\end{tabular} & \begin{tabular}{@{}c@{}}$7.777 \times 10^{47}$\\$(7.700\times10^{47}, 7.855\times 10^{47})$\end{tabular}& $7.770 \times10^{47}$ \\ [1ex] 
 \hline
  12& $10^7$  & \begin{tabular}{@{}c@{}}$3.102 \times 10^{60}$\\$(3.091\times10^{60}, 3.114\times 10^{60})$\end{tabular} & \begin{tabular}{@{}c@{}}$3.083 \times 10^{60}$\\$(3.053\times10^{60}, 3.114\times 10^{60})$\end{tabular} &---\\ [1ex] 
 \hline
  13& $10^7$  & \begin{tabular}{@{}c@{}}$7.523 \times 10^{74}$\\$(7.480\times10^{74}, 7.566\times 10^{74})$\end{tabular} & \begin{tabular}{@{}c@{}}$7.427 \times 10^{74}$\\$(7.353\times10^{74}, 7.502\times 10^{74})$\end{tabular} &---\\ [1ex] 
 \hline
  14& $10^7$  & \begin{tabular}{@{}c@{}}$1.274 \times 10^{91}$\\$(1.263\times10^{91}, 1.285\times 10^{91})$\end{tabular} & \begin{tabular}{@{}c@{}}$1.261 \times 10^{91}$\\$(1.249\times10^{91}, 1.274\times 10^{91})$\end{tabular} & ---\\ [1ex] 
 \hline
  15& $10^7$  & \begin{tabular}{@{}c@{}}$1.724 \times 10^{109}$\\$(1.702\times10^{109}, 1.747\times 10^{109})$\end{tabular} & \begin{tabular}{@{}c@{}}$1.728 \times 10^{109}$\\$(1.710\times10^{109}, 1.745\times 10^{109})$\end{tabular} &---\\ [1ex] 
 \hline
  16& $10^7$  & \begin{tabular}{@{}c@{}}$2.168 \times 10^{129}$\\$(2.113\times10^{129}, 2.224\times 10^{129})$\end{tabular}  & \begin{tabular}{@{}c@{}}$2.211 \times 10^{129}$\\$(2.167\times10^{129}, 2.255\times 10^{129})$\end{tabular} &---\\ [1ex] 
 \hline
  17& $10^8$  & \begin{tabular}{@{}c@{}}$2.8045 \times 10^{151}$\\$(2.7734\times10^{151},2.8355\times 10^{151})$\end{tabular} & \begin{tabular}{@{}c@{}}$2.766 \times 10^{151}$\\$(2.711\times10^{151},2.821\times 10^{151})$\end{tabular} &---\\ [1ex] 
 \hline
  18& $10^8$  & \begin{tabular}{@{}c@{}}$4.2512 \times 10^{175}$\\$(4.1886\times10^{175}, 4.3138\times 10^{175})$\end{tabular} & \begin{tabular}{@{}c@{}}$4.163 \times 10^{175}$\\$(4.038\times10^{175}, 4.288\times 10^{175})$\end{tabular} &---\\ [1ex] 
 \hline
  19& $10^8$  & \begin{tabular}{@{}c@{}}$8.3851 \times 10^{201}$\\$(8.1913\times10^{201}, 8.5789\times 10^{201})$\end{tabular} & \begin{tabular}{@{}c@{}}$8.594 \times 10^{201}$\\$(8.250\times10^{201}, 8.937\times 10^{201})$\end{tabular} &---\\ [1ex] 
 \hline
 20& $10^8$  & \begin{tabular}{@{}c@{}}$2.4433 \times 10^{230}$\\$(2.3326\times10^{230}, 2.5541\times 10^{230})$\end{tabular} & \begin{tabular}{@{}c@{}}$2.263 \times 10^{230}$\\$(2.150\times10^{230}, 2.376\times 10^{230})$\end{tabular}& ---\\ [1ex] 
 \hline
\end{tabular}
\caption{Comparison of our SIS estimator with the estimator by Kuznetsov \cite{kuznetsov2009estimating}.}
\label{tbl: latin square kuznetsov}
\end{table}
.  

Next we compare our method with three conjectures on the asymptotic value of $L_{k,n}$ and $L_n$. The first is a conjecture by Timashov \cite{Timashov2002OnPO}, who used a formula by O’Neil \cite{oNeil1970asymptotics} for the permanent of a random matrix with given row and column sums to  guess the number of Latin squares. Let $L_{k,n}^{Tim}$, $L_n^{Tim}$ be the number of Latin rectangles and Latin squares, respectively, conjectured by Timashov,
\[L_{k,n}^{Tim}=\frac{(2\pi n/e)^{k/2}(1 - k/n)^{n^2 - nk + 1/2} ([n]_k)^{2n}}{n^{kn}}(1+o(1)),\]
\[L_n^{Tim}=\frac{(2\pi)^{3n/2+1}}{2}e^{-2n^2-n/2-1}n^{n^2+3n/2-1}(1+o(1)).\]

As a second conjecture, Leckey, Liebenau and Wormald \cite{LLW} conjectured the following asymptotics
\[L_{k,n}^{LLW}=f\big(\frac{k}{n}\big)\frac{(n!)^k ([n]_k)^{2n}}{e^{\frac{k}{2}}[n^2]_{kn}},\]
where $f(x)$ is a continuous and increasing function on $[0,1]$, with $f(0)=0$ and $f(1)=\frac{\sqrt{2\pi^3}}{e^{7/4}}$, in particular,
\[L_n^{LLW} =\frac{\sqrt{2\pi^3}}{e^{\frac 74}} \frac{(n!)^{3n}}{e^{\frac{n}{2}}(n^2)!}\approx L_n^{Tim}.\]
Actually Timashov \cite{Timashov2002OnPO} also allowed an additional constant term $C(k,n)$, and the asymptotic  of $L_n^{LLW}$ conjectures that $C(k,n)$ is of the form $f\big({k}/{n}\big)$. As both constants are unspecified in general and the forms we have given above seem quite accurate, we will stick with them without the constant.

As the third conjecture, Eberhard, Manners and Mrazovic \cite{EMM} used Maximum Entropy
methods and Gibbs distributions to give the following conjecture on the number of Latin squares
\[L_n^{EMM}=\frac{(n!)^{3n}}{(n^2)!}e^{-\frac{n}{2}+\frac{5}{6}+O(\frac{1}{n})}.\]

Note the difference between the $L_n^{EMM}$ with Timashov's conjecture for Latin squares, $\frac{L_n^{Tim}}{L_n^{EMM}}\sim \frac{\sqrt{2\pi^3}}{e^{25/12}}\approx .975$. 
In order to visualize the estimated number of Latin squares for large $n$ and compare then with these conjectures, we divide the assymptotics and our estimator by the constant $c_n = \frac{(n!)^{3n}}{(n^2)!}e^{-\frac{n}{2}}$. 
Note that  
$\frac{L_n^{EMM}}{c_n}=e^\frac{1}{2} \approx 1.649$. 
 As shown in  Fig. \ref{fig:latin}, our estimators are in favor of Timashov's conjecture for Latin square.

\begin{figure}[!htbp]
\begin{center}
  \includegraphics[scale=0.47]{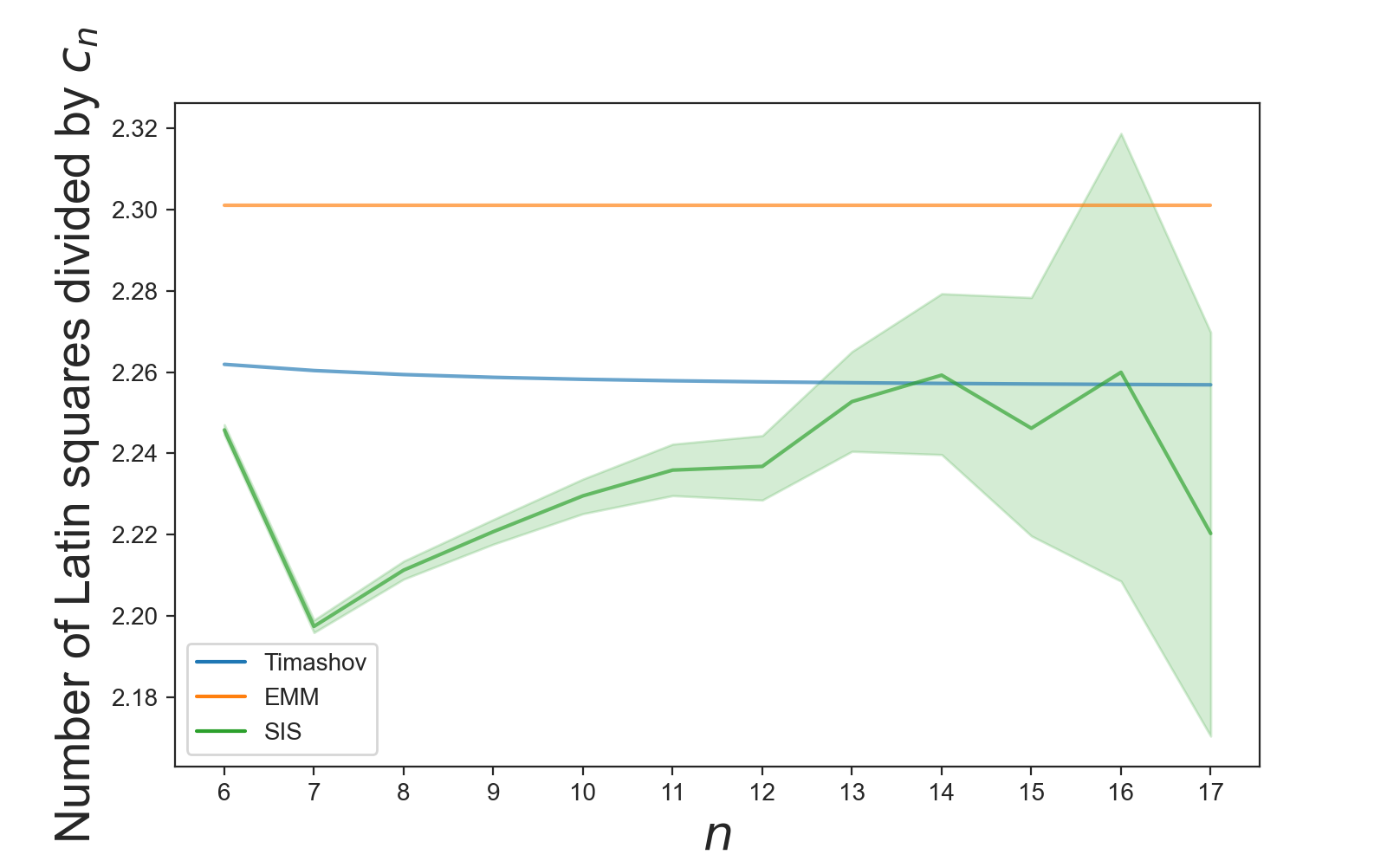}
  \caption{Comparison of SIS estimator $\frac{L_n^{SIS}}{c_n}$ with $\frac{L_n^{EMM}}{c_n}$ and $\frac{L_n^{Tim}}{c_n}$. The $95\%$ confidence intervals are highlighted.}\label{fig:latin}
\end{center}
\end{figure}

Next we fix $k=5$ and compare the conjectures on the number of $5\times n$ Latin rectangles. In order to simplify the plot, we divide each estimator by the factor $c_{k, n} = \frac{(n!)^k ([n]_k)^{2n}}{[n^2]_{kn}}$. We know $\frac{L_{5,n}^{LLW}}{c_{5,n}}=e^\frac{-5}{2} \approx 0.082$, and also, $\frac{L_{k,n}^{Tim}}{L_{k,n}^{LLW}} = e^{-k/12n}$. So,  the difference between conjectures decays with $n$.
In Fig. \ref{fig:latin_rect}, we compare our estimator $L_n^{SIS}$ with both conjectures, and
unlike the case of Latin squares, our results show a better accuracy for Leckey, Liebenau and Wormald's conjecture for Latin rectangles.

\begin{figure}[!htbp]
\begin{center}
  \includegraphics[scale=0.4]{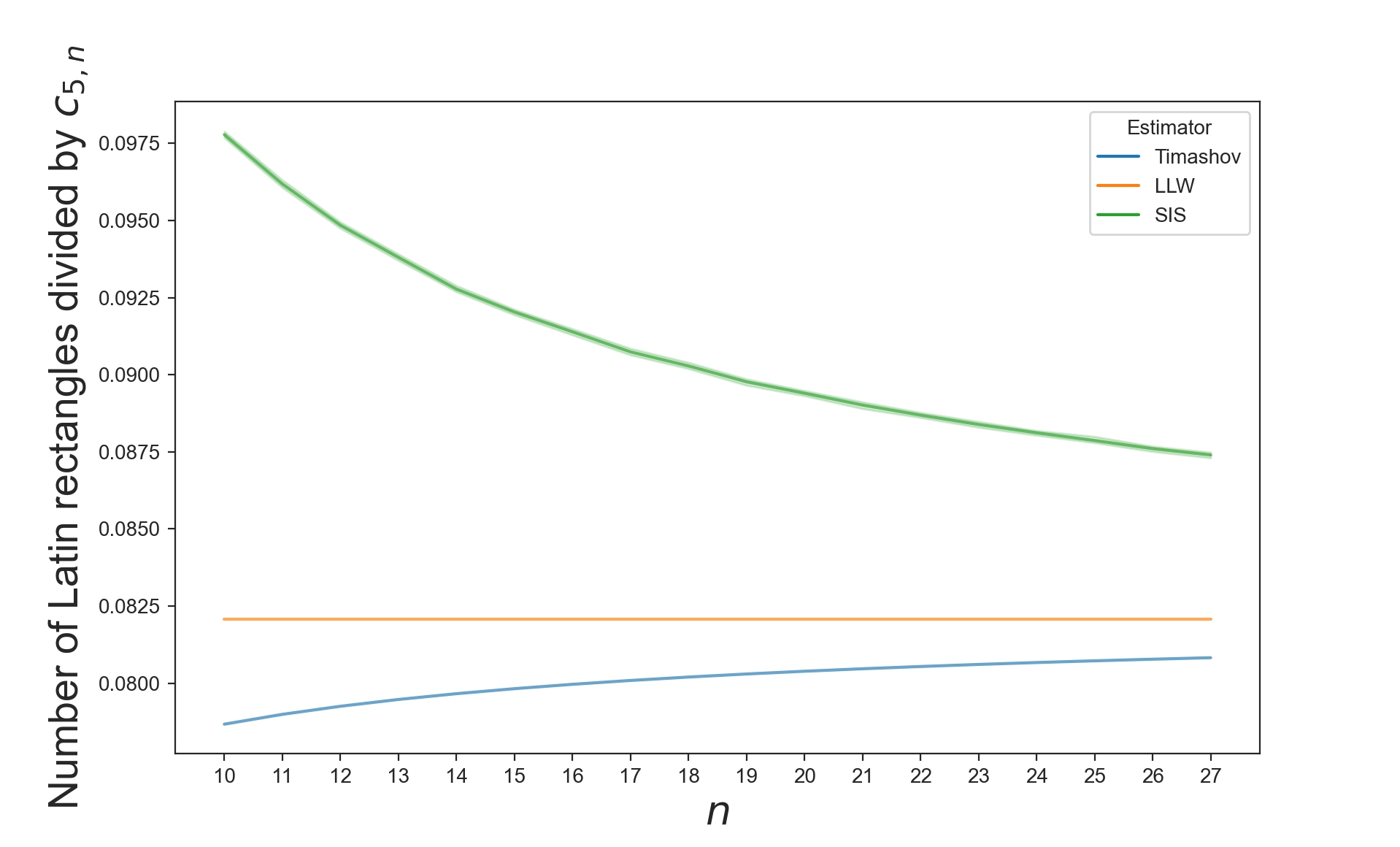}
  \caption{Comparison of $\frac{L_{5,n}^{SIS}}{c_{5,n}}$, $\frac{L_{5,n}^{Tim}}{c_{5,n}}$, and $\frac{L_{5,n}^{LLW}}{c_{5,n}}$. The $95\%$ confidence intervals are too narrow to present, with $(0.0977, 0.0978)$ for $n=10$ and $(0.0873, 0.0874)$ for $n=27$.}\label{fig:latin_rect}
\end{center}
\end{figure}

Finally, we test a conjecture by Cameron on random Latin squares. Call a row of a Latin square odd if its corresponding permutation is odd. The conjecture is as follows: 
\begin{conjecture}[Problem 10 in \cite{cameronschool}]
The number of odd rows of a random Latin square of order $n$ is approximately binomial $Bin(n,\frac{1}{2})$ as $n\to \infty$.
\end{conjecture}
As a special case of this conjecture, H\"aggkvist and Janssen \cite{haggkvist1996all} showed that the number of Latin squares with all even rows is exponentially small. We test the conjecture with sequential importance sampling. Note that each generated Latin square is associated with a weight which is equal to the inverse of the probability of generating that Latin square in importance sampling. Let $\hat o_n$ be the empirical distribution of the number of odd permutations in a Latin square of size $n$ generated by SIS over $10^6$ samples. The first moment Wasserstein distance of $\hat o_n$ with associated SIS weights from the distribution $Bin(n,\frac{1}{2})$
is presented for $n=7$ to $15$  in Table~\ref{tab: odd permutations}. While the distance between two distributions are  small, we do not yet see a decay of the distances with our simulations for $n$ up to $15$.
\begin{table}[!htbp]
    \centering
    \begin{tabular}{|c||c|c|c|c|c|c|c|c|c|}
    \hline
       $n$  & $7$ & $8$ & $9$&$10$& $11$& $12$& $13$& $14$& $15$\\
       \hline 
       \rule{0pt}{3ex}
       $\widetilde{ \mathcal W}$&
       $0.0247$&
$0.0054$&
$0.0030$&
$0.0043$&
$0.0163$&
$0.0140$&
$0.0299$&
$0.0191$&
$0.0279$
       \\
       \hline
    \end{tabular}
    \caption{ The Wasserstein distance ($\widetilde{ \mathcal W}$) of the number of odd rows  in $10^6$ weighted samples of Latin squares using SIS ($\hat o_n$) from $Bin(n,\frac{1}{2})$.}
    \label{tab: odd permutations}
\end{table}

Of course, asymptotics and generating random Latin squares are only one aspect of the problem. The exact calculations indicate that the answers are divisible by surprisingly high powers of $2$. From available data, it is hard to guess at the power or to understand why this should be (see e.g., \cite{Ian} for related open problems). 

\subsection{Card Guessing with Yes/No Feedback}\label{sec: card guessing SIS}

In this section, we apply the repeated estimation of the number of matchings in bipartite graphs to a card guessing experiment. The card guessing experiment starts 
with a well shuffled deck of $N$ cards. A subject guesses the cards one at a time, sequentially, going through the deck. After each guess, the subject is told if their guess is correct or wrong. 

Card guessing experiments with feedback occur in Fisher's tasting tea \cite{fisher:1935}, in evaluation of randomized clinical trials \cite{BlackwellHodges, EfronClinicalTrial}, in ESP experiments \cite{DiaconisESPsearch}, and in optimal strategy in casino card games such as blackjack or baccarat \cite{ethier_levin_2005}. A review is in \cite{DiaconisCompleteFeedback,DiaconisGrahamHolmes}. The recent papers \cite{diaconis2020card,diaconis2020guessing} developed practical  strategies which do perform close to optimal and gave bounds on the expected score under the optimal strategy.

Given a shuffled deck of cards, how should the subject use the feedback to get a high score? 
Extensive numerical work in the literature suggests that greedy (guessing the most likely card given the feedback) is close to optimal. Next we will see that implementing the greedy strategy reduces to evaluating the permanent of a matrix. Then we apply SIS to estimate the expected number of correct guesses using  the greedy strategy for any reasonable deck size. The exact value of greedy was only known for small decks in the literature, and the SIS estimator is close to the known exact values.

To see the relation of card guessing experiments and evaluating permanents let us fix some notation. Consider a deck of $N$ cards, with $d$ distinct values (say, $1, 2, \ldots, d
$) with value $i$ repeated $k_i$ times, so $\sum_1^d k_i = N$. An example is a normal deck of cards $(N = 52)$ with $13$ values each with multiplicity $4$. 

The deck is shuffled and a guessing subject makes sequential guesses. Each time the subject is told if the guess is correct or not. { Consider we want to find the card that the greedy strategy chooses for the next step. Let $a_i$ be the number of cards labeled $i$ that have been in the deck so far but the guessing subject has not guessed them correctly. }
Moreover, let $b_i$ be the number of incorrect guesses where the subject chose card labeled $i$.  As a consequence of the definitions, $\sum_{i=1}^d a_i= \sum_{i=1}^d b_i$.
Let $N(a,b)$ be the number of permutations of a deck of $N' = \sum_1^d a_i$ cards, where cards labeled $1$ are not in first $b_1$ positions and cards labeled $2$ are not in next $b_2$ positions and so on.
Then $N(a,b) = \per(M_{ab})$ with  $M_{ab}$ a zero-one matrix with zero blocks of size $a_1\times b_1,a_2\times b_2,\ldots$ as shown in Figure \ref{fig:zero_block_matrix}.
\begin{figure}[htbp]
\begin{center}
  \includegraphics[scale=0.35]{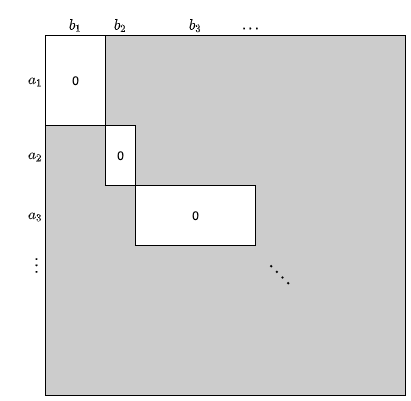}
  \caption{An example of structure of card guessing matrix ($M_{ab}$).}
  \label{fig:zero_block_matrix}
\end{center}
\end{figure}

The number of consistent arrangements with card labeled $i$ in the next position is equal $N(a^{*i},b)$ where $a^{*i}_i = a_i - 1$ and $a^{*i}_j = a_j$ for all $j \neq i$. Thus, the chance that next guess is $i$ equals
$$\frac{N(a^{*i},b)}{N(a,b)}.$$ 
The greedy algorithm guesses $i$ to maximize this ratio. Naively, $d$ permanents must be evaluated at each stage. Some theory allows a simplification. In \cite{ChungPermanent}, the following theorem is proved.
\begin{theorem}\label{thm: greedy alg}
With Yes/No feedback, following an incorrect guess of $i$, the greedy strategy is to guess $i$ for the next guess.
\end{theorem}

Theorem~\ref{thm: greedy alg} shows that the permanent must only be evaluated following a correct guess. Our algorithm for card guessing use the sequential importance sampling algorithm developed in the sections above to do this via Monte Carlo: fix $B$ and generate random matchings in the appropriate bipartite graph $B$ times, weighting each matching by its probability. Use these weighted samples to estimate the chance that the next card is $i$ and choose the maximizing $i^*$.

Throughout, we simulate greedy using `guess card $i$ until told correct' to start. After a `Yes' answer, sequential importance sampling  runs $B$ times as above to estimate the chance of the $d$ possible values for the next card. The most probable is chosen and one keeps guessing this value until `Yes'. Then again begins to simulate, and repeats the procedure above. The plots below compare various values of $B$.

\textbf{ Example 1 ($k=2$ and varying $d$):}
We begin by discussing a long open case. A deck of size $2d$ of composition $1,1,2,2,3,3,\ldots,d,d$. It was open until recently whether the expected number of correct guesses is unbounded in $d$. In \cite{diaconis2020card} it was shown to be bounded by $6$ for all $d$. 
Figure \ref{fig: k=2 guess card} shows some numbers.

\begin{figure}[htbp]
\begin{center}
  \includegraphics[scale=0.19]{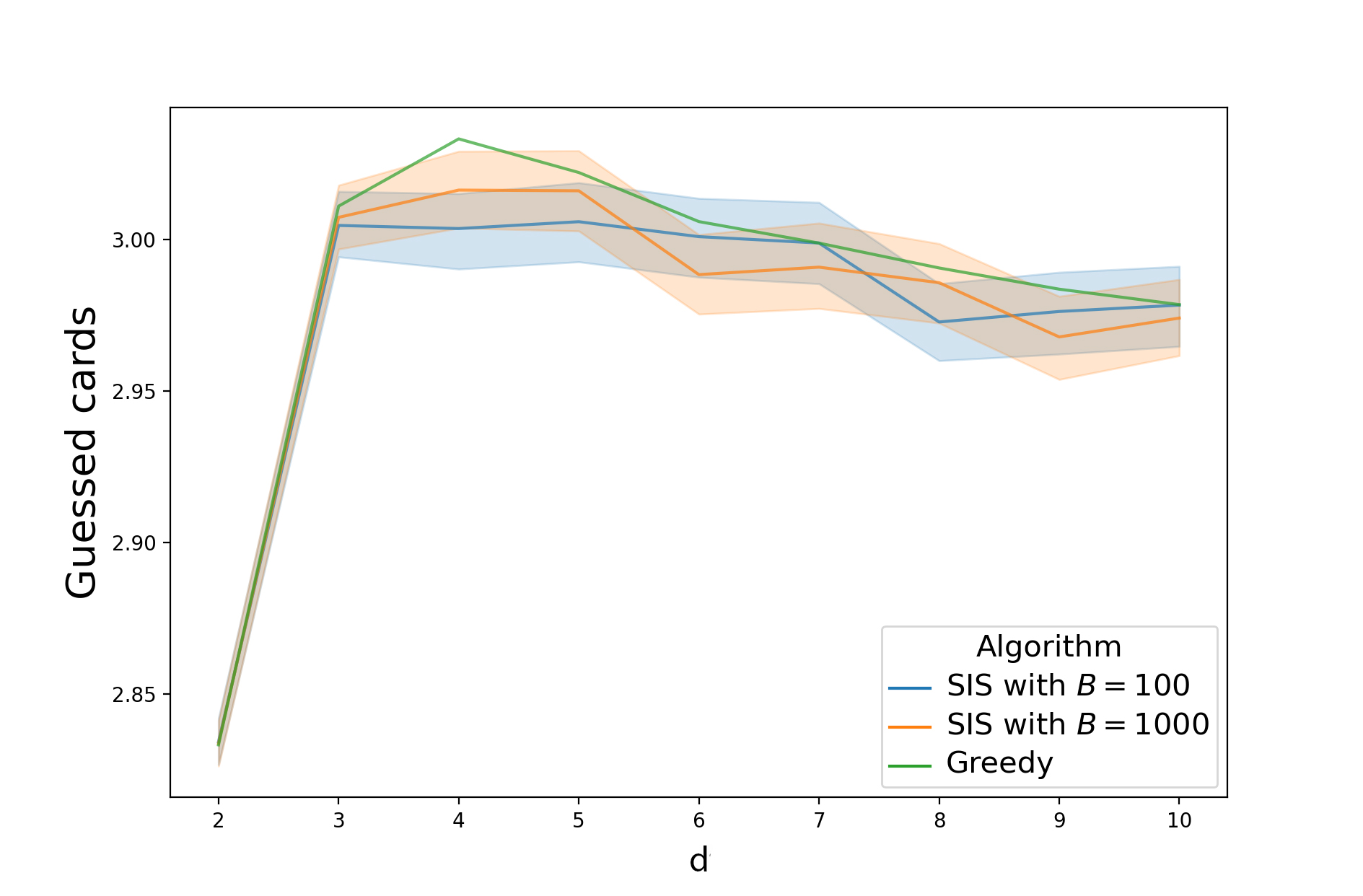}
  \caption{Comparison between different algorithms when $k=2$. The $95\%$ confidence intervals are highlighted.}\label{fig: k=2 guess card}
\end{center}
\end{figure}

\begin{remark}
The figure shows estimates of the expected values with $B=100$ and $B=1000$. The shaded regions are pointwise $95\%$ confidence intervals. Averages are based on $10000$ repetitions. These figures suggest that the lower bound $2.91$ from \cite{diaconis2020guessing} is close to the truth.
\end{remark}
The exact computation of general permanents is a well-known $\#P$-complete problem. However, on reflection, we have found that the permanents of matrices of the form of Figure \ref{fig:zero_block_matrix} can be exactly computed; in linear time in the size of the matrix (see Appendix \ref{sec: exact permanent card guessing}). Having an exact algorithm allows us to check the numerical examples for card guessing. For small values of $d$ the exact expectation of greedy expectations was computed by backward induction for $d=2,3,4,5$ in Table \ref{tab:guesses d=k} and for larger $d$, we use our exact greedy algorithm to estimate the greedy strategy. 

\begin{table}[htbp]
    \centering
    \begin{tabular}{||c||c|c|c||}
    \hline
        $d$ & GREEDY & SIS with $B=100$ & SIS with $B=1000$\\
    \hline\hline
        $2$ & 2.8333 & \begin{tabular}{@{}c@{}}$2.8343$\\$(2.8265, 2.8421)$\end{tabular} & \begin{tabular}{@{}c@{}}$2.8338$\\$(2.8260, 2.8415)$\end{tabular}\\
        \hline
        $3$ & 3.0111 & \begin{tabular}{@{}c@{}}$3.0047$\\$(2.9937, 3.0158)$\end{tabular} & \begin{tabular}{@{}c@{}}$3.0074$\\$(2.9962, 3.0186)$\end{tabular}\\
        \hline
        $4$ & 3.0333 & \begin{tabular}{@{}c@{}}$3.0037$\\$(2.9912, 3.0162)$\end{tabular} & \begin{tabular}{@{}c@{}}$3.0164$\\$(3.0039, 3.0289)$\end{tabular}\\
        \hline
        $5$ & 3.0433 & \begin{tabular}{@{}c@{}}$3.0060$\\$(2.9930, 3.0190)$\end{tabular} & \begin{tabular}{@{}c@{}}$3.0162$\\$(3.0031, 3.0292)$\end{tabular}\\
        \hline
        $10$ & \begin{tabular}{@{}c@{}}$2.9786$\\$(2.9773, 2.9799)$\end{tabular} & \begin{tabular}{@{}c@{}}$2.9784$\\$(2.9649, 2.9918)$\end{tabular} & \begin{tabular}{@{}c@{}}$2.9741$\\$(2.9606, 2.9876)$\end{tabular}\\
        \hline
        $20$ & \begin{tabular}{@{}c@{}}$2.9585$\\$(2.9571, 2.9598)$\end{tabular} & \begin{tabular}{@{}c@{}}$2.9589$\\$(2.9454, 2.9724)$\end{tabular} & \begin{tabular}{@{}c@{}}$2.9557$\\$(2.9423, 2.9691)$\end{tabular}\\
        \hline
    \end{tabular}
    \caption{Greedy versus SIS estimator with $95\%$ confidence interval for $k=2$}
    \label{tab:guesses d=k}
\end{table}

\textbf{Example 2 ($d=k$)}: The classical ESP experiment had $d=k=5$. The theorems in \cite{diaconis2020card} and \cite{diaconis2020guessing} worked for fixed $k$ and large $d$ or fixed $d$ and large $k$. There are virtually no results for guesses when both $k$ and $d$ are large. Using SIS any reasonable deck sizes are now accessible. Fig. \ref{fig:CardGuessingM=n} shows results using sequential importance sampling for $d=k$ for $2 \leq d = k \leq 18$. The algorithm in Appendix \ref{sec: exact permanent card guessing} allows computing the number of correct guesses by the greedy algorithm for larger $d=k$. Note that the number of correct guesses by greedy depends on the order of the cards in the shuffled deck, and as a result we show confidence intervals both for greedy and SIS in Fig. \ref{fig:CardGuessingM=n}. 
\begin{figure}[htbp]
  \includegraphics[scale=0.16]{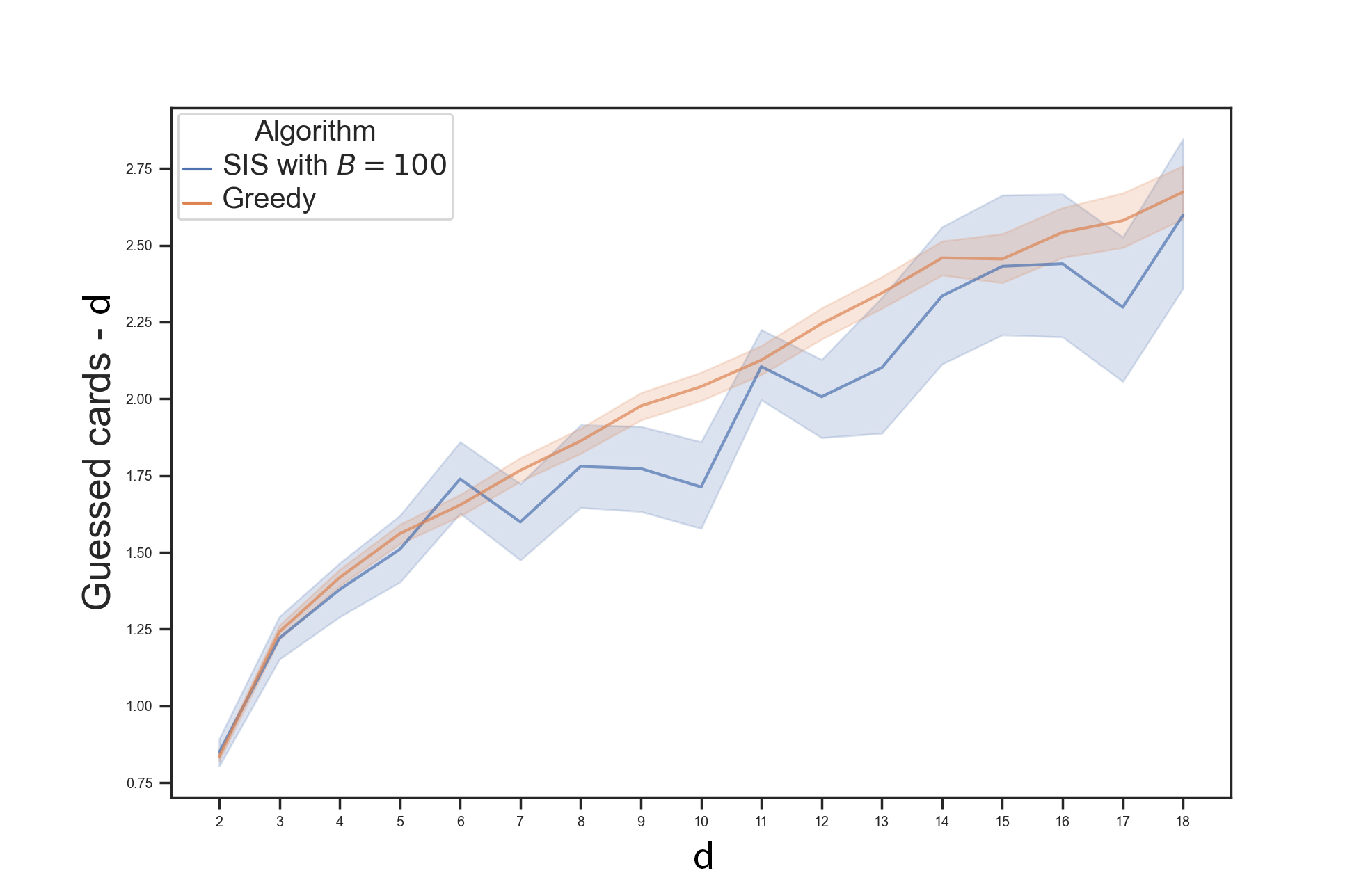}
  \caption{Comparison between different algorithms when $d=k$. The $y$-axis represents 'the number of guessed cards'$-d$. The $95\%$ confidence intervals are highlighted.}\label{fig:CardGuessingM=n}
\end{figure}

\begin{remark}
From this data, the greedy algorithm guesses at least $d$ cards correctly. From previous data, it was not clear if the excess over $k$ even got as large as $2$ (!). We now believe that for $k=d$, the experiment of the greedy strategy is $d + \sqrt{d} + o(\sqrt{d})$ (see Fig. \ref{fig:nMinusSqrtn} ).
\end{remark}

\begin{figure}[htbp]
\begin{center}
  \includegraphics[scale=0.13]{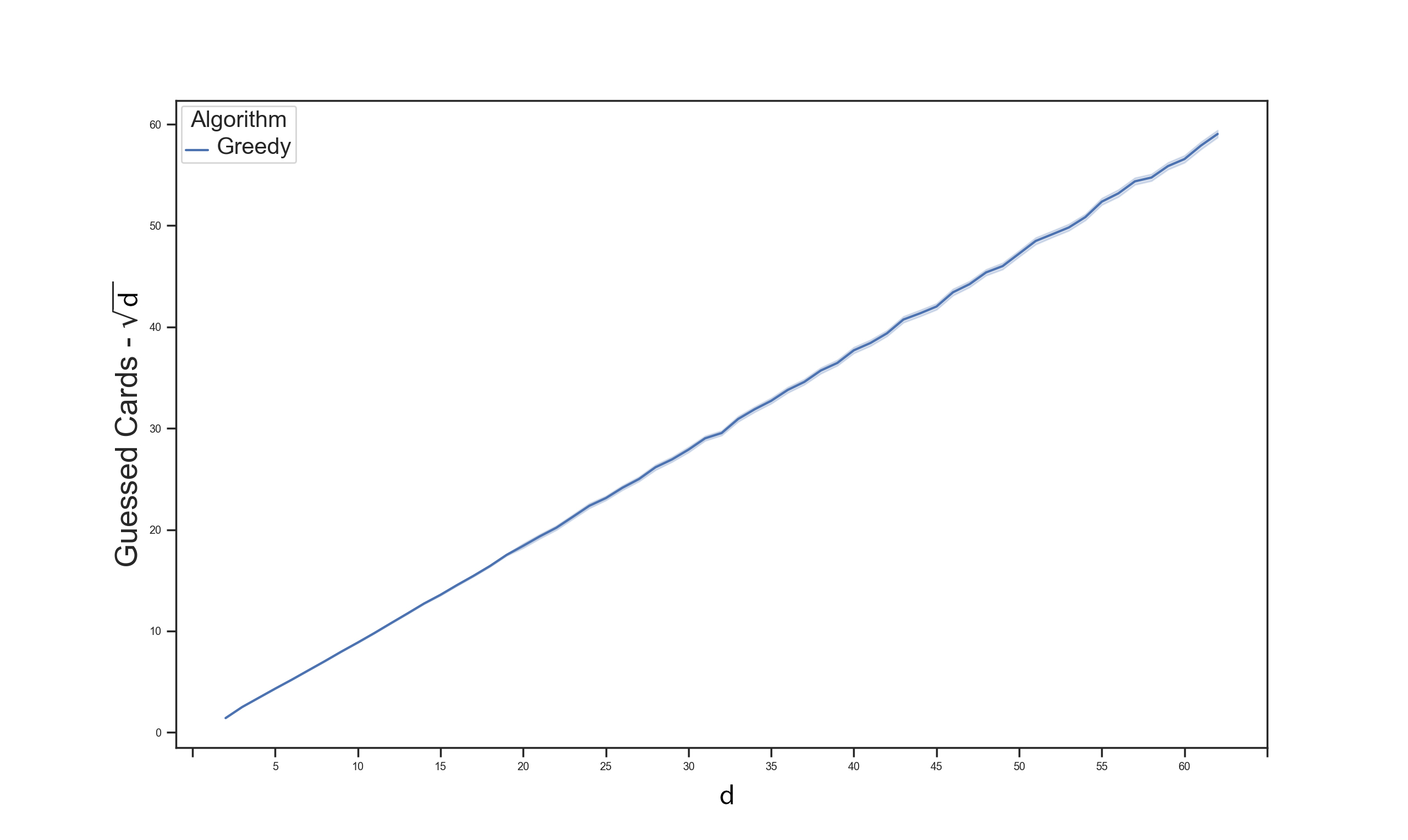}
  \caption{The plot of the `guessed cards' $- \sqrt{d}$ when $d=k$. The $95\%$ confidence intervals are  too narrow to present, with $(2.8213, 2.8484)$ for $d=2$ and $(66.5565, 67.2839)$ for $d=62$.}\label{fig:nMinusSqrtn}
\end{center}
\end{figure}

\subsection{Stochastic Block Models}\label{sec: SIS for SBM}
The purpose of this section is to show the effect of doubly stochastic scaling on the convergence and variance of our estimator. The stochastic block model is a generative model for random graphs with community structures. It is widely used for statistical analysis of community structures and social networks. 

Let $n$ be the number of vertices and $C_1, C_2, \ldots, C_k$ be the clusters. Also, assume that $P$ is a $k \times k$ matrix for edge probability between clusters. The probability that $u \in C_i$ is connected to $v \in C_j$ is equal to $P_{ij}$. We can similarly define stochastic block models for bipartite graphs. If we have $k$ clusters in one part and $r$ clusters in the other part, then the probability matrix will be a $k \times r$ matrix which $P_{ij}$ shows the probability that a vertex from the cluster $i$ in the first part is connected to a vertex from the cluster $j$ of the second part.

We focus on the case that each part of the bipartite graph $G(X,Y)$ has only two clusters. Let $X_1, X_2$ be clusters in first part and $Y_1, Y_2$ be clusters in the other part. Hence, the matrix $P$ will be an $2 \times 2$ matrix. Let us assume that $P_{11} = P_{12} = P_{21} = p$ and $P_{22} = q$. Intuitively, for large $p$ and small $q$, most of the edges of a perfect matching would be selected from edges between $X_1$, and $Y_2$ or from edges between $X_2$, $Y_1$. In other words, If we select many edges between $X_1$ and $Y_1$, then there could be a near-perfect matching between $X_2$, $Y_2$ which has a low probability. Therefore, using the algorithm without doubly stochastic scaling of the adjacency matrix is not efficient. 

In Fig. \ref{fig:convergence}, we sampled one graph with the given values of $p$ and $q$ when $n=20$. Then we used SIS (with/without doubly stochastic scaling) estimator to count the number of perfect matchings in the sampled graph.
As shown in the figures, SIS with doubly stochastic scaling converges faster. Given that both estimators are unbiased, an important property for comparison is the standard deviation of the estimator. In Table \ref{tbl: variance}, we compare the standard deviation of the SIS with doubly stochastic scaling and without doubly stochastic scaling. In this table each row correspond to one sampled graph from the stochastic block model.  In some cases the standard deviation of the SIS without doubly stochastic is 10 times higher than the SIS with doubly stochastic scaling.

\begin{remark}
Given $p,q\in[0,1]$, it is possible to calculate the expected number of perfect matchings over all graphs drawn from doubly stochastic block model. However, the variance on the number of perfect matchings is large. Therefore, if we do not fix one sampled graph it is hard to differentiate the efficiency of one estimator over the other. 
\end{remark}
\begin{figure}[!htbp]
\begin{center}
\subfloat[$p=1, q=0.1$.]{\includegraphics[width = 3in]{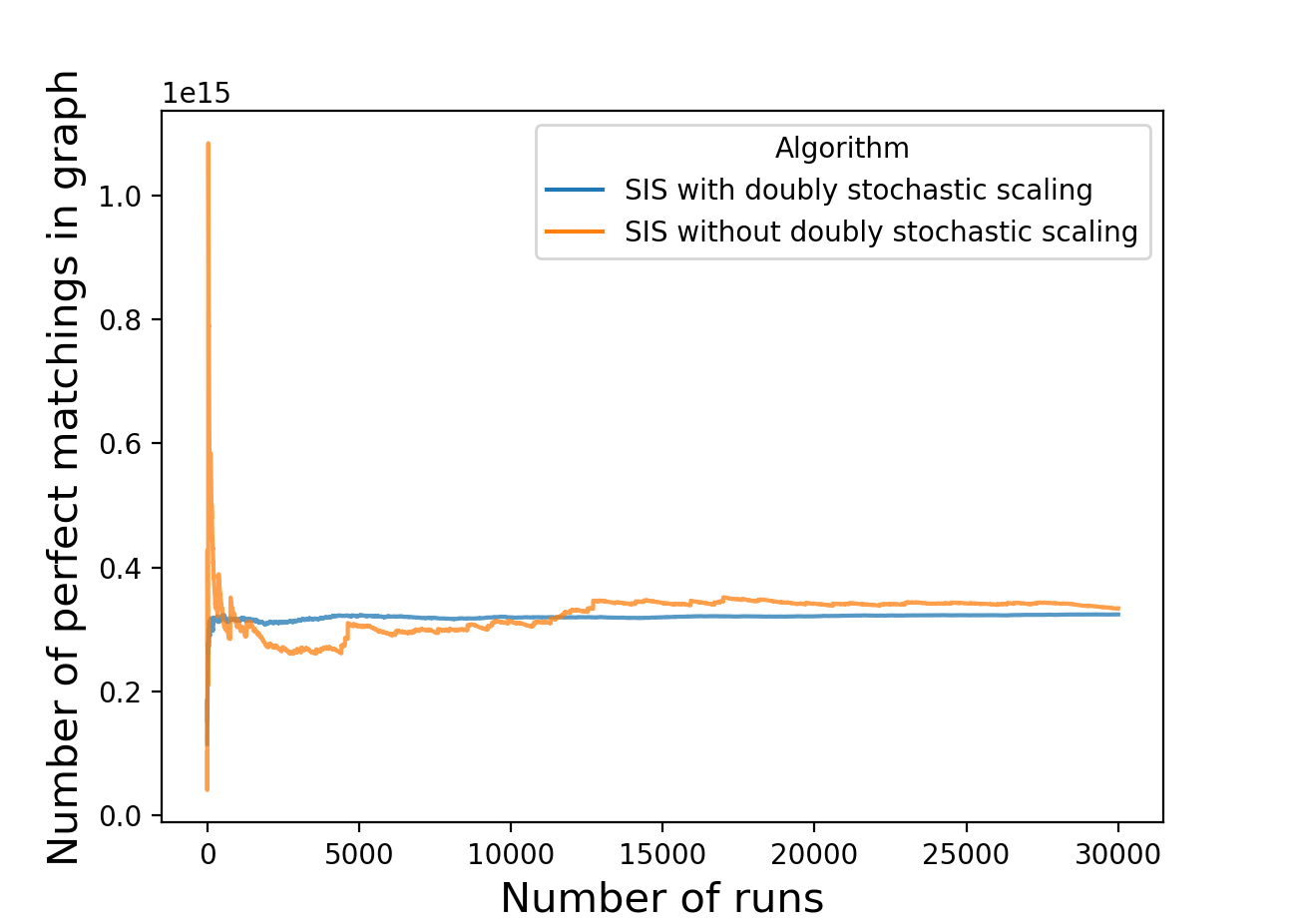}} 
\subfloat[$p=0.9, q=0.2$.]{\includegraphics[width = 3in]{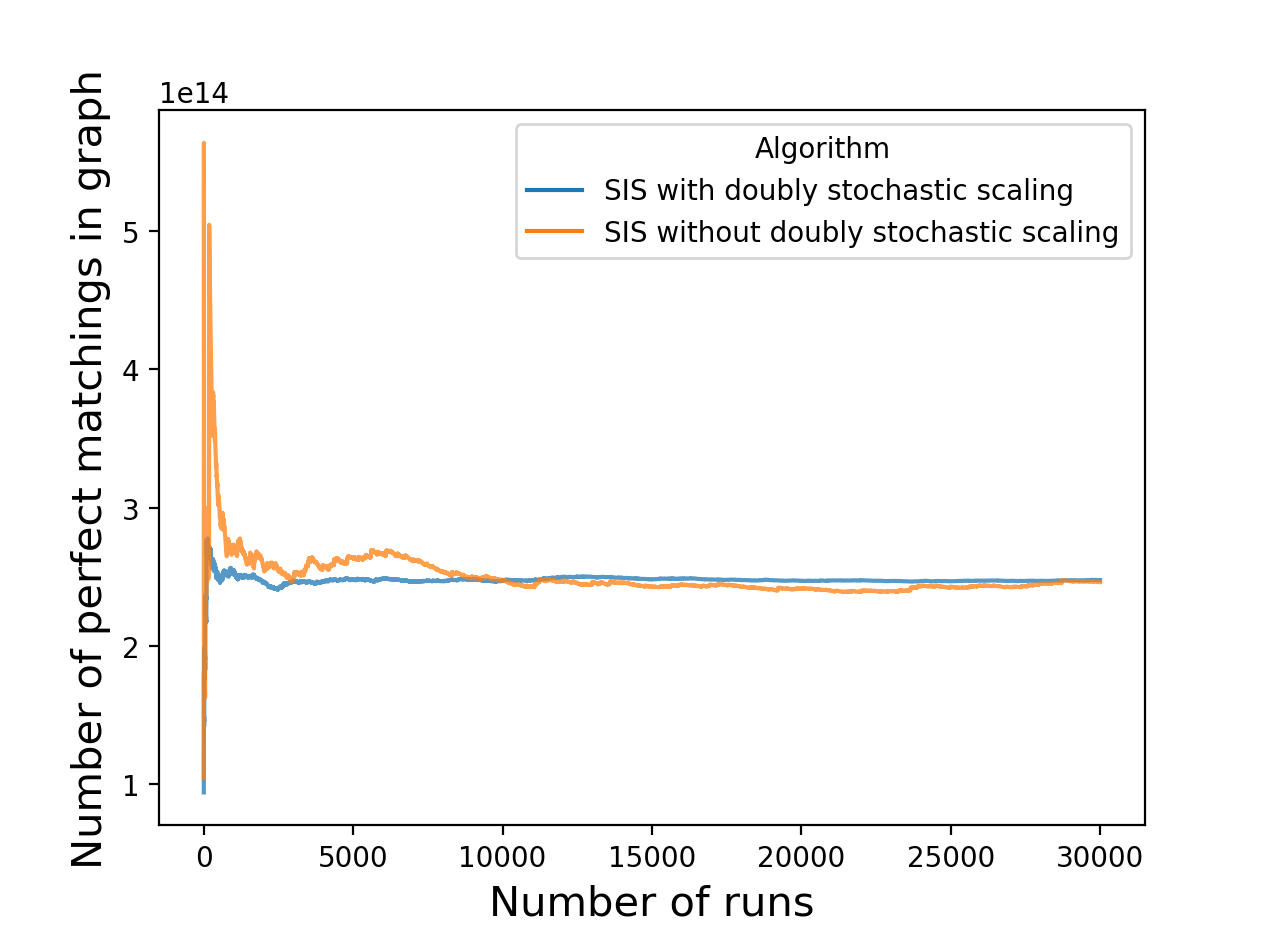}}
\caption{Comparing the convergence of the SIS with doubly stochastic scaling and without doubly stochastic scaling for different $q$ and $p$ when $n=20$. Here, the $x$-axis shows the number of runs of algorithm and the $y$-axis shows SIS estimator so far.}
\label{fig:convergence}
\end{center}
\end{figure}

\begin{table}[!htbp]
 \begin{tabular}{||c|c|c|c|c|c|c||} 
 \hline
 \rule{0pt}{15pt}
 $n$ & $p$ & $q$ & Runs & \begin{tabular}{@{}c@{}}Standard deviation of SIS\\with doubly stochastic\end{tabular} & \begin{tabular}{@{}c@{}}Standard deviation of SIS\\without doubly stochastic\end{tabular} & \begin{tabular}{@{}c@{}}Mean of SIS\\with doubly stochastic\end{tabular}\\ [0.5ex] 
 \hline\hline
$20$ & $1$ & $0.1$ & $10^5$ & $3.0699 \times 10^{14}$ & $2.2088 \times 10^{15}$ & 
$3.2606 \times 10^{14}$\\ 
 \hline
$20$ & $1$ & $0.2$ & $10^5$ & $1.9577 \times 10^{15}$ & $8.2753 \times 10^{15}$ & 
$2.3127 \times 10^{15}$\\ 
\hline

$20$ & $0.9$ & $0.1$ & $10^5$ & $3.4717 \times 10^{13}$ & $1.9276 \times 10^{14}$ & 
$3.0495 \times 10^{13}$\\ 
\hline
$20$ & $0.9$ & $0.2$ & $10^5$ & $2.5417 \times 10^{14}$ & $9.0992 \times 10^{14}$ & 
$2.4765 \times 10^{14}$\\
\hline

$20$ & $0.8$ & $0.1$ & $10^5$ & $3.4055 \times 10^{12}$ & $1.9168 \times 10^{13}$ & 
$2.5553 \times 10^{13}$\\ 
\hline
$20$ & $0.8$ & $0.2$ & $10^5$ & $3.0373 \times 10^{13}$ & $1.0374 \times 10^{14}$ & 
$2.6025 \times 10^{13}$\\

\hline
\end{tabular}
\caption{Comparison of standard deviation of SIS with or without doubly stochastic scaling.}
\label{tbl: variance}
\end{table}

 \section*{Acknowledgements}
 
 The authors thank Sourav Chatterjee, for helpful feedback on the early version of the manuscript and the proof of Lemma \ref{lm: rho concentration}, and Jan Vondr\'ak for discussions on concentration inequalities. Also, we would like to thank Nick Wormald and Fredrick Manners for bringing us up to speed about Latin squares. We thank our anonymous reviewers for their insightful comments and suggestions.
   
   Yeganeh Alimohnammadi, Mohammad Roghani and Amin Saberi are supported by NSF grant CCF1812919.

\bibliographystyle{imsart-number} 
\bibliography{ref}       

\begin{appendix}
\section{An Implementation of Algorithm \ref{alg: SIS-matching}} \label{sec: fast implmnt SIS}
 Given a graph $G(X,Y)$ and a partial matching $\overline M$, recall that an edge $e\in E(G)$ is $\overline M$-extendable, if there exists a perfect matching that contains $\overline M\cup\{e\}$. We then simply say an edge is extendable if there exists a perfect matching containing it.
 We implement Algorithm \ref{alg: SIS-matching} by finding extendable edges fast.
The main idea presented in Section~\ref{sec: extendable edges} is to use Dulmage-Mendelsohn decomposition \cite{matching-theory}, to create a directed graph from $G$, so that an edge is extendable if and only if both of its endpoints are in the same strongly connected component. 
 
In addition, we need to keep track of extendable edges as the Algorithm \ref{alg: SIS-matching} decides on adding an edge to the partial matching constructed so far. An algorithm that takes linear time to update the decomposition at each step is given in Section~\ref{sec: fast SIS modify decomposition}.

\subsection{Creating a Directed Acyclic graph and Maintaining Extendable Edges}\label{sec: extendable edges}
 To find all extendable edges, note that the union of any two perfect matchings of $G$ creates a cycle decomposition on the nodes.
Let $M$ be a perfect matching. Construct a directed graph $D_G(X,Y)$ by directing all edges of $G(X,Y)$ from $Y$ to $X$ and adding a directed copy of $M$ from $X$ to $Y$ (see Fig. \ref{fig: SCC decomposition}). An edge is extendable if both of its endpoints are in the same strongly connected component. 

\begin{algorithm}[htbp]
\SetAlgoLined
\textbf{Input:} Bipartite graph $G(X,Y)$, a perfect matching $M$.

Construct $D_G(X,Y)$.

Find all strong connected components (SCC) of $D_G(X,Y)$.

\For{any edge $e\in E(G)$}{
\eIf{endpoints of $e$ are in the same SCC}{$e$ is extendable.}{$e$ is not extendable.}
}

 \textbf{Output:} return extendable edges.
 \caption{Finding all Extendable Edges}\label{alg:finding-edges}
\end{algorithm}

\begin{figure}
\centering
\begin{tikzpicture}[-,>={Stealth[round,sep]},shorten >=1pt,auto,node distance=1cm]

    \node[main node] (1) {$x_0$};
    
    \node[main node] (3) [right =of 1]{$x_1$};
    \node[main node] (2) [right =of 3]{$x_2$};
    \node[main node] (4) [right =of 2]{$x_N$};
    \node[main node] (5) [below =of 1]{$y_0$};
    
    \node[main node] (7) [right =of 5]{$y_1$};
    \node[main node] (6) [right =of 7]{$y_2$};
    \node[main node] (8) [right =of 6]{$y_N$};

    \node at ($(2)!.5!(4)$) {\ldots};
    \node at ($(6)!.5!(8)$) {\ldots};
    \draw[matching edge] (1) -- (5);
     \draw[edge] (5.830) -- (1.970);
    \draw[edge] (6) -- (1);
     \draw[edge] (8) -- (2);
     
     \draw[edge] (7) -- (2);
     \draw[edge] (5) -- (4);
     \draw[matching edge] (2) -- (6);
     
     \draw[edge] (6.830) -- (2.970);
     \draw[matching edge] (4) -- (8);
     \draw[edge] (8.830) -- (4.970);
     \draw[matching edge] (3) -- (7);
     \draw[edge] (7.830) -- (3.970);
\end{tikzpicture}
\caption{An example of constructing $D_G(X,Y)$, where the black edges are $E(G)$, and the red edges are in  the matching $M$.}\label{fig: SCC decomposition}
\end{figure}
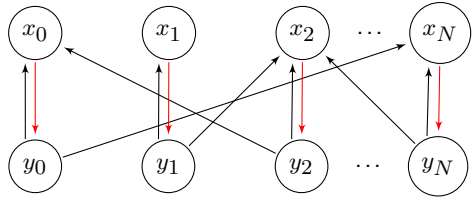

\begin{lemma}
\label{lemma-extendable-edges}
Given a bipartite graph $G(X,Y)$ and a perfect matching $M$, Algorithm \ref{alg:finding-edges} finds all the extendable edges in $O(n+m)$ time, where $m$ is the number of edges, $n$ is the number of vertices of $G$. 
\end{lemma}

\begin{proof}
First, we prove the output of the algorithm is the set of all extendable edges. For that purpose, we show an edge $e=(u,v)$ is extendable if and only if  both $u$ and $v$ are in the same strongly connected component of $G$.  First, assume that $e$ is an edge in some perfect matching $M'$ (it can be equal to $M$).  Edges of $M\cup M'$ create a directed cycle decomposition on $D_G(X,Y)$, since edges of cycles in $M\cup M'$ alternate between $M$ and $M'$, and  $M$ and $M'$ are directed in opposite directions. So, in this case, both ends of $e$ are in the same strongly connected component.

Now, assume $e=(u,v)$ and $u$, $v$ appear in the same strong connected component. Then we show that $e$ is extendable. Since $u$ and $v$ are in the same SCC, there exists a directed cycle $C=(u,v,a_1,a_2,\ldots,a_{2n},u)$ in $D_G$.  If $e$ is a part of $M$ we are done. So assume $e\not\in M$. Therefore, edges of $C$ alternate between edges of $M$ and edges not in $M$. By replacing edges of $M$ that appear in the cycle $C$, with edges of $C$ that are not in $M$, we still get a perfect matching. In other words, remove $(v,a_1)$, $(a_3,a_4),\ldots, (a_{2n-1},a_{2n}),(a_{2n},u)$ from $M$, and add $(u,v)$, $(a_2,a_3),\ldots, (a_{2n-2},a_{2n-1})$ to construct $M'$, then $e$ is an edge of perfect matching $M'$.

To analyze the algorithm's running time, note that finding the strongly connected components of $D_G$ can be done in $O(n+m)$ by \cite{Tarjan1972DepthFirstSA}. 
\end{proof}

\subsection{Keeping Track of Extendable Edges }\label{sec: fast SIS modify decomposition}
Now that we know how to find extendable edges, we can implement Step 5 of Algorithm \ref{alg: SIS-matching}.
Given a graph $G$ and one of its perfect matchings $M'$, first construct $D_G(X,Y)$ from $M'$ as in Section \ref{sec: extendable edges}. Then after each vertex is matched, we need to update $D_G(X,Y)$ so that each extendable edge of the current partial matching can be found by the SCC decomposition of $D_G(X,Y)$. 

To do so, when an edge $e$ is added to the partial matching, we consider two cases. First, if $e$ is directed from $X$ to $Y$, i.e., it is already in the perfect matching, we can only remove $e$ from $D_G(X,Y)$ and the cycle decomposition of the rest of the graph indicate new extendable edges. In the second case, when $e$ is directed from $Y$ to $X$,  find a directed cycle that it appears in (which exists since $e$ is extendable). Note that if we reverse all edges in this cycle, we get a new graph $D_G(X,Y)$ such that $e$ is directed from $X$ to $Y$. Then as in the first case, we can remove edge $e$.

\begin{algorithm}[H]\label{SIS-matching}
\SetAlgoLined
\textbf{Input:} a bipartite graph $G(X,Y)$, and a nonnegative matrix $Q_{n\times n}$.

Draw a random permutation $\pi$ on $X$.

Let $M=\emptyset$,  $p_{\pi,Q}(M)=1$.

Run Algorithm \ref{alg:finding-edges}, and let $S$ be the SCC decomposition, and $D_G$ be the digraph.

\For{$i$ from $1$ to $n$}{
Find the set of neighbors of $\pi(i)$ that are in the same SCC with it (call it $N_{\pi(i)}$).

Find $Q_{\pi(i)}[N_{\pi(i)}]$, the restriction of row $\pi(i)$ of $Q$ to indices in $N_{\pi(i)}$.

Let $i^*$ be the random index in $N_{\pi(i)}$ drawn with the probability  proportional to $Q_{\pi(i)i^*}$.

$M=M\cup{(\pi(i),i^*)}$.

$p_{\pi,Q}(M)= p_{\pi,Q}(M)\times \frac{Q_{\pi(i),i^*}}{\sum_{j\in N_{\pi(i)}}Q_{\pi(i),j}} $.

\If{$(\pi(i),i^*)$ is directed from $Y$ to $X$ in $D_G$}{Find a directed cycle $C$ which contains the edge $(\pi(i),i^*)$  in $D_G$.

Reverse all edge of $C$.
}

Delete $\pi(i)$ and $i^*$ from $D_G$.

Find SCC decomposition of $D_G$.

}

 \textbf{Output:} $M, p_{\pi,Q}(M)$.
 \caption{ A Detailed Implementation of Algorithm \ref{alg: SIS-matching}}
\end{algorithm}
\begin{remark}
To run Algorithm \ref{alg:finding-edges} in the beginning, we can find one perfect matching using the Micali-Vazirani algorithm (see e.g., \cite{MicaliVazirani}).
\end{remark}

\begin{lemma}
Suppose $G$ is a bipartite graph which has at least one perfect matching, and let $Q$ be a nonnegative matrix such that $Q_{ij}>0$ for all $(i,j)\in E(G)$. Then Algorithm \ref{SIS-matching} always returns a perfect matching.
\end{lemma}
\begin{proof}
It is enough to show that for all $i$, $N(\pi(i))$ is the set of all extendable adjacent edges of $\pi(i)$.
We prove this by induction. At each step of the algorithm, all edges directed from $X$ to $Y$ form a perfect matching.  Then  result follows by the correctness of Algorithm \ref{alg:finding-edges}, proved in Lemma \ref{lemma-extendable-edges}.

The base case of induction is true by construction. At step $i$ of the algorithm, when we match $\pi(i)$ to $j$ there exists a perfect matching $M'$ that $e=(\pi(i),j)\in M'$. Similar to the proof of Lemma \ref{lemma-extendable-edges}, we can construct matching $M'$ so that its edges are the same with matching $M$, except edges that appear in cycle $C$. This is what happens in Algorithm \ref{SIS-matching}: It removes the copy of edges in $M\cap C$ that are directed from $X$ to $Y$, if $e\not\in M$, and adds all edges of $M'\cap C$ from $X$ to $Y$. So, at the end edges that are directed from $X$ to $Y$ are edges of a perfect matching that contains $e$, which proves the induction step.
 \end{proof}
\begin{prop}\label{prop: runtime}
Let $G$ be a bipartite graph with $n$ vertices and $m$ edges. Then the Algorithm \ref{SIS-matching} runs in $O(nm)$ time to sample a perfect matching.
\end{prop}
\begin{proof}
First, we need to find one perfect matching to build a directed graph on top of it. Finding a perfect matching takes at most $O(n^{1/2}m)$ by using \cite{MicaliVazirani}. 

Each iteration of the algorithm requires finding a cycle. For that purpose, one can use a depth-first-search which needs $O(n+m)$ operations. Also, finding the strongly connected component decomposition can be done with $O(n+m)$ operations \cite{Tarjan1972DepthFirstSA}. Therefore, Algorithm \ref{SIS-matching} can be done in $O(n(n+m))$ operations. It is worth noting that sampling from rows of $Q$ can be done in linear time, which is bounded by the computation time of each step. 
\end{proof}

Note that the Dulmage-Mendelsohn decomposition has been already used to find extendable edges when the partial matching is an empty-set (see 
\cite{maximal-matchable-edges}). However, in this section, we extended the idea to update the Dulmage-Mendelsohn decomposition and find new extendable edges after each step of the SIS algorithm without modifying the decomposition a lot.

\section{An Example for the Necessity of the Doubly Stochastic Scaling }\label{sec:counter example matching}
The main modification of our algorithm in comparison to the algorithm in \cite{diaconis2019randomized} is using the doubly stochastic scaling of the adjacency matrix as the input of Algorithm  \ref{alg: SIS-matching}. We show this scaling is necessary by giving an example where the uniform sampling of edges at each round needs an exponential number of samples while sampling according to the doubly stochastic scaling of the adjacency matrix only needs a linear number of samples.

\begin{prop}\label{matching-counter-example}
Let $G_n(X,Y)$ be a bipartite graph with $X=\{x_0,x_1,\ldots, x_n\}$ and $Y=\{y_0,y_1,\ldots, y_n\}$. Assume in $G_n$, $x_i$ is adjacent to $y_i$ for all $i$, and $x_0$ and $y_0$ are adjacent to all vertices of the other side. Let $\mu$ be the uniform distribution over the set of perfect matchings and $\nu$ be the sampling distribution of Algorithm \ref{alg: SIS-matching} with $Q$ equal to an all ones matrix. If $\rho=d\nu/d\mu$ then there exist constants $\epsilon,c>0$ such that $\mathbb{P}_{\nu}(\log \rho(X)>L+\epsilon n)>c$. 
\end{prop} 
 Note that this result along with Theorem \ref{thm: num-samples} shows that an exponential number of samples are needed in this class of graphs if we sample edges uniformly at random at each step of Algorithm \ref{alg: SIS-matching}.
 Before proving this result, we state the analogous result that shows 
\begin{prop}\label{matching-counter-example-DS}

Let the graph $G_n$ be as defined in Proposition \ref{matching-counter-example}. Given a constant $\epsilon>0$, there exists a constant $C>0$ such that  $Cn\epsilon^{-2}$ samples  of Algorithm  \ref{alg: SIS-matching} is enough to give an $(1-\epsilon)$-approximation of the number of perfect matchings.
\end{prop}

\begin{proof}[Proof of Proposition \ref{matching-counter-example}]
Define the following notations for perfect matchings in $G$, 
$$M_0=\{(x_i,y_i):1\leq i\leq n\},$$
and for $1\leq i\leq n$,
$$M_i=\{(x_0,y_i),(x_i,y_0)\}\cup\{(x_j,y_j):j\not\in\{0,i\}\}.$$
Consider $M_0$, and let $\pi(i)$ be the location of $x_i$ in the permutation $\pi$. All $x_i$ that come before $x_0$, have two edges to choose between, and for all nodes that come after $x_0$, there is only one edge left. For $x_0$ itself, there are $n-\pi(0)+1$ edges available that any one of them can extend to a perfect matching. Therefore, $p_{\pi,Q}(M_0)=\frac{2^{-\pi(0)}}{n-\pi(0)+1}$, which implies
$$\mathbb{E}_{\pi}\big(\log_2 p_{\pi,Q}(M_0)\big)=-\sum_{i=0}^n \frac{1}{n+1}(i+\log_2(n+1-i))=-\frac{n}{2}-\frac{\log_2((n+1)!)}{n+1}.$$

Now, for matching $M_i$ with $1\leq i\leq n$, consider two cases to find the sampling probability. Each vertex with $\pi(i)<\pi(0)$ has two choices for pairing, and otherwise it has only one choice. So, when $\pi(i)<\pi(0)$, $p_{\pi,Q}(M_i)=2^{-\pi(i)}$. For the case $\pi(i)>\pi(0)$, all vertices before $\pi(0)$ have two choices, and $\pi(0)$ have $n-\pi(0)+1$ choices. Vertices after that have only one extension to a perfect matching. Therefore,
$p_{\pi,Q}(M_i)=\frac{2^{-\pi(0)}}{n-\pi(0)+1},$ which again implies,
\begin{align*}
    \mathbb{E}_{\pi}\big(\log_2 p_{\pi,Q}(M_i)\big)&=-\sum_{j=0}^n\Big( \frac{n+1-j}{n(n+1)}(j+\log_2(n+1-j))+\sum_{k=0}^{j}\frac{k}{n(n+1)}\Big)\\
    &=-\sum_{j=0}^n\Big( \frac{n+1-j}{n(n+1)}(j+\log_2(n+1-j))+\frac{j(j+1)}{2n(n+1)}\Big)\\
    &= -\sum_{j=0}^n\frac{j}{n}+ \frac{\log_2(n+1-j)}{n}+\frac{(n+1-j)\log_2(n+1-j)}{n(n+1)}
\end{align*}
The last equality shows that there exists constants $C_1,C_2\geq 0$ such that
\[  -\frac{n}{3}+C_1\log_2 n\leq \mathbb{E}_{\pi}\big(\log_2 p_{\pi,Q}(M_i)\big)\leq -\frac{n}{3}+C_2\log_2 n.\]
By taking expectation over the uniform distribution, $\mu$, on all perfect matchings $M_0,M_1,\ldots,M_n$, we see that there exists constants $C_1',C_2'\geq 0$ such that
$$ \frac n3+C_1'\log_2 n\leq \mathbb{E}_{\mu,\pi}\big(\log_2 p_{\pi,Q}(M)\big)\leq \frac n3+C_2'\log_2 n,$$
where $1/2\geq C\geq 1/3$.
Recall that $\rho(M)=\frac{1}{np_{\pi,Q}(M)}$. Then for small enough $\epsilon$ , 
$$ \mathbb{P}_\nu(\log_2\rho(M)\geq \frac n3+\epsilon n)\geq \frac 12-\epsilon.$$
the latter inequality is true, because for all $M_i$s,  if $\pi(0)\geq n/2+\epsilon n$ then there exists $\epsilon>0$ such that $-\log_2(p_{\pi,Q}(M_i))\geq n/2$.
\end{proof}

\begin{proof}[Proof of Proposition \ref{matching-counter-example-DS}]
The doubly stochastic scaling of the adjacency matrix, $Q_A$, has $1/n$ in its first row and column and $1-\frac{1}{n}$ on all other diagonal entries and zero everywhere else. 

Define $M_i$'s as in the proof of Proposition \ref{matching-counter-example}. If $\pi(0)=i$ with probability $(1-\frac1n)^i\frac{1}{n-i}$ the algorithm generates matching $M_0$. In this case, if we let $q=\frac1n$, then $$\log p_{\pi,Q_A}(M_0)=i\log(1-q)+\log(q).$$ 

For other matchings $M_i$, where $i\geq 1$, with probability $(1-q)^{\min(\pi(i),\pi(0))}q/(1-qi)$ we get the matching $M_i$ with $$\log p_{\pi,Q_A}(M_i)=(\min(\pi(i),\pi(0))-1)\log(1-q)+\log(q).$$
Therefore,
$$\max_{M_i,\pi}\{-\log p_{\pi,Q_A}(M_i)\}\leq n\log(\frac{n}{n-1})+\log(n)\leq 2+\log n.$$
Then if we let $L=\mathbb E_\mu\rho(\log n)$ as in Theorem \ref{thm: num-samples},
 $$\mathbb{P}_\nu\big(\log p_{\pi,Q_A}(M_i)\geq L+\log n\big)=0,$$
 which proves the statement by  Theorem \ref{thm: num-samples}.
\end{proof}

\section{Computing Permanent of Zero-Blocked Matrices}\label{sec: exact permanent card guessing}
Given sequences $a=(a_1,\ldots,a_r)$ and $b=(b_1,\ldots,b_r)$, let $M_{ab}$ be a zero-one matrix, such that zeros forms diagonal blocks of sizes $a_1 \times b_1, a_2 \times b_2, \ldots, a_r \times b_r$ (see Figure \ref{fig:zero_block_matrix}). We show that it is possible to compute $\per(M_{ab})$ in a linear time in the number of entries of the matrix. As a result, one can compute the greedy strategy for card guessing, as discussed in Section \ref{sec: card guessing SIS}.

\begin{lemma}
Let $M_{ab}$ be defined as above. Suppose  $M_{ab}$ is an $n \times n$ matrix. Then $\per(M_{ab})$ can be computed in $O(n^2)$ time.
\end{lemma}
\begin{proof}
To describe the algorithm for computing $\per(M_{ab})$, we need the following notations. Let $f(i,j)$ be the number of the ways that we can select $j$ zeros in the first $i$ blocks such that no two zeros are in the same row or same column. Since there are $r$ blocks in total, $f(r,j)$ shows the number of the ways that we can select $j$ zeros in $M_{ab}$ such that no two zeros are in the same row or the same column. Then we can formulate the permanent of $M_{ab}$ by the inclusion-exclusion principle, 
$$
\per(M_{ab}) = \sum_{j=0}^{n} (-1)^j (n - j)!f(r,j).
$$

Therefore, if we can compute $f(i,j)$ for all $0 \leq i \leq r$ and $0 \leq j \leq n$ in $O(n^2)$ time, we can compute the permanent of the matrix $M_{ab}$ in $O(n^2)$ time. 

We know that the size of the $i^{th}$ block is $a_i \times b_i$. The number of ways to choose $k$ zeros in this block in such a way that no two zeros are in the same row or column is 
${a_i \choose k}\frac{b_i!}{(b_i-k)!}$. Let $q_i = \min(a_i, b_i)$. Then $k$ can take any value in $\{0,1,\ldots,q_i\}$. Therefore,
\begin{equation}\label{eq: dp update}
f(i, j) = \sum_{k = 0}^{q_i} f(i-1,j-k) {a_i \choose k}\frac{b_i!}{(b_i-k)!}.
\end{equation}

Now, we claim that we can find the values of $f(i,j)$ for all $0\leq i\leq r, 0\leq j\leq n$ in $O(n^2)$ time. First, note that we can compute all the values of ${i \choose j}$ for all $0 \leq j \leq i \leq n$ using Pascal's rule in $O(n^2)$ time. We run a procedure to find the values of $f(i,j)$ using recursive induction. First, initialize all values of $f(i,j)$ to zero except for $f(0, 0) = 1$. Then iterate over $i = \{1,2,\ldots, r\}$ and $j = \{1, 2, \ldots, n\}$ in an increasing order and update the value of an entry, based on the previously calculated values. To compute $f(i,j)$, by \eqref{eq: dp update}  we need $O(\min(a_i,b_i))$ operations. Hence, for a fixed $i$ and all $0 \leq j \leq n$, computing $f(i,j)$ takes $O(n \cdot \min(a_i,b_i))$ operations. As a result, the total number of operations for computing all values of $f$ are
$$
\sum_{i=0}^r O(n \cdot \min(a_i,b_i)) = O(n^2),
$$
which proves the claim. Therefore, computing the permanent of $M_{ab}$ can be done in $O(n^2)$ time. 
\end{proof}

\end{appendix}
\end{document}